\documentclass[microtype]{gtpart}
\usepackage[parfill]{parskip}
\usepackage{graphicx}
\usepackage{amsmath}
\usepackage{amstext}
\usepackage{amsfonts}
\usepackage{amsthm}
			\usepackage{amssymb}
\usepackage{subfig}
\usepackage{hyperref}
\usepackage{rlepsf}

\title{Geometric Composition in Quilted Floer Theory}
\author{Yank\i\ Lekili} 
\address{University of Cambridge, Cambridge UK}
\email{yl319@cam.ac.uk}
\author{Max Lipyanskiy}
\address{Columbia University, New York, NY USA}
\email{mlipyan@math.columbia.edu}

\newtheorem{theorem}{Theorem}
\newtheorem{lemma}[theorem]{Lemma}
\newtheorem{proposition}[theorem]{Proposition}
\newtheorem{definition}[theorem]{Definition}
\newtheorem{corollary}[theorem]{Corollary}

\newtheorem{remark}[theorem]{Remark}

\newcommand{\f}[1]{\mathbb{#1}}

\newcommand{\del}{\partial}
\newcommand{\delbar}{\bar{\partial}}

\newcommand{\Ar}{\mathbb{R}}
\newcommand{\ep}{\epsilon}
\newcommand{\I}{\text{index}}

\newcommand{\QED}{\vspace{-.31in}\begin{flushright}\qed\end{flushright}}

\begin{document}

\begin{abstract} 

We prove that Floer cohomology of cyclic Lagrangian correspondences is invariant under
transverse and embedded composition under a general set of assumptions.

\end{abstract}
\maketitle

\section{Introduction}

\subsection{Lagrangian correspondences and geometric composition}

Given two symplectic manifolds $(M_1, \omega_1)$, $(M_2, \omega_2)$ a Lagrangian
correspondence is a Lagrangian submanifold $L \subset (M_1 \times M_2, -\omega_1
\oplus \omega_2)$.  These are the central objects of the theory of holomorphic quilts
as developed by Wehrheim and Woodward in \cite{WW}. Consider two Lagrangian
correspondences $L_i \subset (M_{i-1} \times M_i, - \omega_{i-1} \oplus \omega_i )$
for $i=1,2$.  Let  \[ \Delta = \{ (x,y,z,t) \in M_0\times M_1\times M_1\times M_2 \ | \ 
y=z \} \]  If $L_1 \times L_2$ is transverse to $\Delta$, we may form the fibre product
$L_1\times_{M_1} L_2\subset M_0\times M_1\times M_1 \times M_2$ by intersecting $\Delta$ with
$L_1\times L_2$.  If the projection $L_1\times_{M_1} L_2 \rightarrow M_0\times M_2$ is an
embedding, we say that $L_1$ and $L_2$ are composable and $L_1\circ L_2$ is naturally a Lagrangian submanifold of $M_0\times M_2$ and
is called the (geometric) composition of $L_1$ and $L_2$.   As a point set one has \[
L_1 \circ L_2 = \{\ (x,z) \in M_0\times M_2 \ |\ \ \exists y \in M_1\ \ \text{such
that}\ \ (x,y) \in L_1\ \ \text{and}\  \ (y,z) \in L_2 \} \]

\subsection{Floer cohomology of a cyclic set of Lagrangian correspondences} 

A cyclic set of Lagrangian correspondences of length $k$ is a set of Lagrangian
correspondences $L_i \subset M_{i-1} \times M_i$ for $i=1,\ldots,k$ such that
$(M_0, \omega_0)  = (M_k, \omega_k)$.

Given a cyclic set of Lagrangian correspondences, Wehrheim and Woodward in \cite{WW}
define a Floer cohomology group $HF(L_1,\ldots,L_{k})$ (see Section \ref{definition}
below for a review). This can be identified with the Floer cohomology group of the
Lagrangians $$L_{(0)} = L_1 \times L_3 \times \ldots \times L_{k-1} \text{\ and \ } L_{(1)}=
L_2\times L_4 \times \ldots \times L_k $$ in the product manifold $\underbar{M} =
M^-_0 \times M_1\times M^-_2 \times \ldots \times M_{k-1}$ if $k$ is even. If $k$ is odd, one
inserts the diagonal $\Delta_{M_0} \subset M_0^- \times M_0 = M^-_{k+1} \times M_0$ to
get a cyclic set of Lagrangian correspondences with even length. (We denote by $M^-$ the symplectic manifold $(M,-\omega)$ where $\omega$
is the given symplectic form on $M$).  Under appropriate assumptions on the underlying Lagrangians, one expects an isomorphism
\[ HF(L_0,\ldots,L_r,L_{r+1},\ldots,L_{k-1}) \simeq HF(L_0,\ldots, L_r \circ L_{r+1},
\ldots, L_{k-1}) \]

when $L_r$ and $L_{r+1}$ are composable. The main goal of the present work is to prove such an isomorphism under a rather general set of assumptions. Such an isomorphism should exist whenever the Floer cohomology groups on either side can be defined.
For instance, let us discuss this isomorphism in the aspherical case. For this we need to introduce some notation. Namely, given two
transverse Lagrangians $L, L' \subset (M,\omega)$, we consider the path space: 
$$\mathcal{P}= \mathcal{P}(L,L') = \{\gamma :[0,1] \to M | \ \gamma(0) \in L \ , \
\gamma(1) \in L' \}$$
Now pick $x_0 \in L\cap L'$ to be the constant path on a fixed component $\mathcal{P}_{0}$ of $\mathcal{P}$. Then given any path $\gamma \in
\mathcal{P}_{0}$, we can pick a smooth homotopy $\gamma_t$ such that
$\gamma_0 =x_0$ and $\gamma_1=\gamma$. Then consider the action functional : 
\begin{eqnarray*} 
	\mathcal{A} : \mathcal{P}_{0} \to \f{R}  \\ 
\gamma \to \int_{[0,1]^2} \gamma_t^* \omega 
\end{eqnarray*}

This is not always well-defined, because in general it depends on the choice of the
homotopy $\gamma_t$. However, under various topological assumptions, it is possible
to avoid this dependence. 

A simple case of the main result in this paper is the following statement:

\begin{theorem} 
\label{isomorphism} 
Given a cyclic set of compact connected orientable Lagrangian correspondences
$L_1,\ldots,L_k$ in compact symplectic manifolds $(M_0,\omega_0),\ldots,
(M_k,\omega_k)$ such that for some $r$, $L_r$ and $L_{r+1}$ can be composed, suppose that the following
topological properties hold: 
 \begin{equation}\begin{array}{c}
	 \displaystyle \text{For any\ \ } v : S^1 \times [0,1] \to \underbar{M}\text{\ \ such that\ \ } v|_{S^1 \times \{0\}} \subset L_{(0)} \text{\ \ and \ \ } v_{S^1 \times \{1\}} \subset L_{(1)} \\ 
\displaystyle \int v^* \omega_{\underbar{M}} = 0\\

\end{array}
\end{equation}
Then, 
\[ HF(L_0,\ldots,L_r,L_{r+1},\ldots,L_{k-1}) \simeq HF(L_0,\ldots, L_r \circ L_{r+1},
\ldots, L_{k-1}) \]

\end{theorem}

The assumption (1) is used to avoid bubbling in various moduli spaces and to ensure that the action functional
\begin{equation*}
	\begin{array}{l}
	 \displaystyle \mathcal{A} : \mathcal{P}_0 (L_{(0)},L_{(1)}) \to \f{R}  
\end{array}
\end{equation*} 

is single valued (on any of its path components) which ensures that the Floer differential squares
to zero. These assumptions are already required for the Floer cohomology groups
considered above to be well-defined. One could replace them with assumptions of
similar nature but not dispose of them altogether.

The analogous result under positive monotonicity assumptions was proved earlier by Wehrheim and Woodward in \cite{WW}. The difficulty in extending their proof to our setting is the
fact that the strip shrinking argument in \cite{WW} might give rise to certain {\it
figure-eight bubbles} for which no removal of singularities is known.  Our proof of
the theorem above does not involve strip shrinking and does not give rise to
figure-eight bubbles. Applying the idea used for the proof of Theorem
\ref{isomorphism}, we will give an alternative proof of the positive monotone case considered in \cite{WW}. 

\begin{theorem}\label{positive}(positively monotone case) 
Let $L_1,\ldots,L_k$ be a cyclic set of compact orientable Lagrangian correspondences
in compact connected symplectic manifolds $(M_0,\omega_0),\ldots,
(M_k,\omega_k)$ such that for some $r$, $L_r$ and $L_{r+1}$ can be composed. Let
$\tau > 0 $ be a fixed real number. Suppose that the following
topological properties hold: 
\begin{equation}\begin{array}{c}
\displaystyle \text{For any\ \ } v : S^1 \times [0,1] \to \underbar{M}\text{\ \ such that\ \ } v|_{S^1 \times \{0\}} \subset L_{(0)} \text{\ \ and \ \ } v_{S^1 \times \{1\}} \subset L_{(1)} \\ 
\displaystyle \int v^* \omega_{\underbar{M}} = \tau I_{\text{Maslov}}(v^* L_{(0)} ,
v^*L_{(1)}) \\
\displaystyle \text{The minimal Maslov index for disks in\ } \pi_2(\underbar{M},
L_{(0)}) \text{\ and\ } \pi_2(\underbar{M},L_{(1)}) \text{\ is\ } \geq 3. \\
\end{array}
\end{equation}

Then, 
\[ HF(L_0,\ldots,L_r,L_{r+1},\ldots,L_{k-1}) \simeq HF(L_0,\ldots, L_r \circ L_{r+1},
\ldots, L_{k-1}) \]
\end{theorem}

Note that we only require monotonicity for the annuli with boundary on $L_{(0)}$ and
$L_{(1)}$ which makes the group on the left well-defined. However, it is easy to see
that the corresponding monotonicity relation for the group on the right hand side
follows from this. Furthermore, via the natural map from $\pi_2(\underbar{M}) \to \pi_2(\underbar{M}; L_{(0)}, L_{(1)})$ the hypotheses of the theorem implies the following monotonicity of the
symplectic manifolds $M_i$, which determines the monotonicity constant $\tau$.
\[ [\omega_{M_i}] =
\tau c_1(TM_i) \text{\ for all\ } i .\]

Note also that when $\tau=0$, the symplectic manifolds are exact and
necessarily non-compact, thus one needs to assume convexity properties
at infinity (as for example in \cite{Seidel}) in order to ensure compactness of
various moduli spaces. The proof is simpler in the exact case. Indeed, the
hypothesis of Theorem \ref{isomorphism} are satisfied, hence this case is
covered by the previous result.

Finally, we extend the argument to the (strongly) negatively monotone case which is
needed for the application to quilted Floer homology of broken fibrations (\cite{lekili}) . Recall that for $[u] \in \pi_2(M)$, the expected dimension of the moduli space of unparametrized holomorphic spheres in class $[u]$ is given by $2 ( \langle c_1(TM), [u] \rangle + \text{dim}(M) -3)$ and for $[u] \in \pi_2(M, L)$, the expected dimension of the moduli space of unparametrized holomorphic disks in class $[u]$ is given by $\mu_L([u]) + \text{dim}(M)-3$ (where $\mu_L$ is the Maslov homomorphism). In the strongly negative case, we require these numbers to be sufficiently negative in order to avoid bubbling in $0,1$ and $2$-dimensional moduli spaces.

\begin{theorem}\label{negative}(strongly negative monotone case) 
Given a cyclic set of compact connected orientable Lagrangian correspondences
$L_1,\ldots,L_k$ in compact connected symplectic manifolds $(M_0,\omega_0),\ldots,
(M_k,\omega_k)$ such that for some $r$, $L_r$ and
$L_{r+1}$ can be composed. Let $\tau < 0$ be a fixed real number. Denote $\text{dim\ }M_i = 2m_i$. Suppose that the following
topological properties hold for all $i=0,\ldots,k$: \begin{equation}\begin{array}{c}
 \displaystyle \text{For any\ \ } v : S^1 \times [0,1] \to \underbar{M}\text{\ \ such that\ \ } v|_{S^1 \times \{0\}} \subset L_{(0)} \text{\ \ and \ \ } v_{S^1 \times \{1\}} \subset L_{(1)} \\ 
\displaystyle \int v^* \omega_{\underbar{M}} = \tau I_{\text{Maslov}}(v^* L_{(0)} ,
v^*L_{(1)}) \\
%\displaystyle \text{For any\ \ } v:S^1 \times[0,1] \to (M_{r-1} \times M_r \times
%M_r \times M_{r+1}; L_r \times L_{r+1} , (L_r \circ L_{r+1}) \times \Delta) \\
%\displaystyle \int v^* \omega_{M_{r-1}\times M_r \times M_r \times
%M_{r+1}} = \tau I_{\text{Maslov}}(v^* (L_r \times L_{r+1}) , v^*((L_r \circ L_{r+1})
%\times \Delta)) \\
\end{array}
\end{equation}
\begin{equation}\begin{array}{l}
\displaystyle \text{If\ } \int u^*(\omega_i) > 0 \text{\ for\ } [u] \in \pi_2(M_i),
\text{\ then \ } \langle c_1(TM_i), [u] \rangle < -m_i+2.  \\ 
\displaystyle \text{If\ } \int u^*(-\omega_i \oplus \omega_{i+1}) > 0 \text{\ for\ }
[u] \in \pi_2(M_i \times M_{i+1}, L_{i+1}), \\ \ \ \ \ \ \ \ \ \ \ \text{\ then \ }
\mu_{L_{i+1}}([u]) <
-(m_i+m_{i+1}) +1 . \\ 
\end{array}
\end{equation}

Then, 
\[ HF(L_0,\ldots,L_r,L_{r+1},\ldots,L_{k-1}) \simeq HF(L_0,\ldots, L_r \circ L_{r+1},
\ldots, L_{k-1}) \]

\end{theorem}

In all of the cases, the main idea is to construct a particular homomorphism $$\Phi:
HF(L_0,\ldots,L_r,L_{r+1},\ldots,L_{k-1}) \rightarrow HF(L_0,\ldots, L_r \circ
L_{r+1}, \ldots, L_{k-1})$$  Once $\Phi$ is constructed, a simple energy argument
shows that $\Phi$ is an isomorphism.  

The main motivation for proving Theorem \ref{negative} is an application to an
explicit example. Namely, it is used to get rid of a technical
assumption in the proof of an isomorphism between Lagrangian matching invariants and
Heegaard Floer homology of $3$-manifolds, which appeared in a previous work of the
first author \cite{lekili}. Our main construction was also used in $\cite{manwood}$ in order to prove the topological invariance of their symplectic construction of a Floer homology group which they conjecture to be isomorphic to a version of instanton Floer homology.

In order to avoid repetition, we will not give all the
details involved in the definition of holomorphic quilts and Floer cohomology of a cyclic
set of Lagrangian correspondences. The more comprehensive discussion of foundations of
this theory is available in \cite{WW}.

\paragraph{\bf Acknowledgments:} We would like to thank Denis Auroux, Ciprian
Manolescu, Sikimeti Ma'u, Tomasz Mrowka, Timothy Perutz and Chris Woodward for helpful
comments on an early draft of this paper. The first author is thankful for the support
of Mathematical Sciences Research Institute and Max-Planck-Institut f\"ur Mathematik
during the preparation of this work. The second author thanks Columbia University and the Topology RTG grant. We are grateful to the anonymous referee whose careful reading and critical comments helped improve the text significantly.

\section{Morphisms between Floer cohomology of Lagrangian correspondences }

\subsection{Chain complex of a cyclic set of Lagrangian correspondences} 
\label{definition}

Let us recall that when $L_{(0)}$ and $L_{(1)}$ intersect transversely the chain complex $CF(\underbar{\em{L}})$ associated with $\underbar{\em{L}} = (L_1,\ldots,L_k)$ is the
freely generated group over a base ring $\Lambda$ by the generalized intersection points
$\mathcal{I}(\underbar{\em{L}})$ where 
\[ \mathcal{I}(\underbar{\em{L}}) = \{\underbar{x}=(x_1,\ldots, x_k)\ | \  (x_k,x_1) \in
L_1,(x_1,x_2)\in L_2,\ldots, (x_{k-1},x_k) \in L_k \} \] 

The role of $\Lambda$ here is no different than its role in the usual Lagrangian Floer
cohomology. We will mostly take $\Lambda$ to be $\f{Z}_2$ (or more generally Novikov
rings over a base ring of characteristic $2$) in order to avoid getting into sign
considerations. The full discussion of orientations in this set-up appeared in
\cite{WWorient}, from which one expects that under assumptions on orientability of
the relevant moduli spaces (say when $L_i$ are relatively spin), our results still hold over $\f{Z}$. 

More generally, we can choose Hamiltonian functions $H_i : [0,\delta_i] \times
M_i \to \f{R}$ and perturb $\underbar{\em{L}}$ with a Hamiltonian isotopy on
each $M_i$ to ensure that $L_{(0)}$ and $L_{(1)}$ intersect transversely, so
that $\mathcal{I}(\underbar{\em{L}})$ is a finite set. It is an easy lemma to
show that for a generic choice of $(H_i)_{i=1,\ldots,k}$, transversality of
$L_{(0)}$ and $L_{(1)}$ holds after the perturbation, in particular the set
$\mathcal{I}(\underbar{\em{L}})$ is finite (see \cite{WWc} page 7). From now
on, we will always assume that $L_{(0)}$ and $L_{(1)}$ are perturbed into
general position and we will take $H_i \equiv 0$ for all $i$, so that
$\mathcal{I}(\underbar{\em{L}})$ consists of a finite set of generalized
intersection points as defined in the beginning of this section.

Next, to define the differential on $CF(\underbar{\em{L}})$, for each $i=1,\ldots, k$, we choose a compatible
almost complex structure $J_i$ on $M_i$ and extend
the definition of the Floer differential to our setting in the following way. Let
$\underbar{x}, \underbar{y}$ be generalized intersection points in
$\mathcal{I}(\underbar{\em{L}})$. We define the moduli space of finite energy quilted
holomorphic strips connecting $\underbar{x}$ and $\underbar{y}$ by
 \begin{eqnarray*} 
\mathcal{M}(\underbar{x},\underbar{y}) = \{ u_i : \f{R} \times [0,\delta_i] \to M_i |
\ \delbar_{J_i, H_i} u_i := \del_{s} u_i + J_i (\del_t u_j - X_{H_i}(u_i)) = 0 ,\\
 E(u_i) := \int u_i^*\omega_i - d(H_i(u_i)) dt < \infty \\ \text{lim}_{s\to-\infty}
u_i(s,\cdot) = x_i , \text{\ lim}_{s\to +\infty} u_i(s,\cdot) = y_i \\
 (u_i(s, \delta_{i}), u_{i+1}(s,0)) \in L_{i+1}\  \text{for all}\ i=1,\ldots k \} /
\f{R} 
\end{eqnarray*}

Under monotonicity assumptions (as in Theorem \ref{positive}), it is proven in \cite{WW} (with corrections from \cite{WWL}) that, given
$(\delta_i)_{i=1,\ldots,k}$ and $(H_i)_{i=1,\ldots,k}$, there is a Baire second
category subset of almost complex structures $(J_i)_{i=1,\ldots,i=k}$ for which these
moduli spaces are cut out transversely and compactness properties of the usual Floer
differential carry over. It is straightforward to check that the same result holds
when we replace the monotonicity assumptions by the set of assumptions in the
statement of Theorem \ref{isomorphism} (for more details, see the proof of Theorem
5.2.3 in \cite{WW}). As we check in the proof of Theorem \ref{negative}, the
assumptions of Theorem \ref{negative} also gives rise to well-defined moduli spaces.
Therefore, in either case one can define the Floer differential for a cyclic set of
Lagrangian correspondences by :
\[ \del \underbar{x} = \sum_{\underbar{y} \in \mathcal{I}(\underbar{\em{L}})}
\# \mathcal{M}(\underbar{x},\underbar{y}) \underbar{y} \]
where $\#$ means counting isolated points modulo 2. 

\begin{remark} If one has the additional choices in place so that the moduli
spaces $\mathcal{M}(\underbar{x},\underbar{y})$ are oriented (cf. \cite{Seidel}, \cite{WWorient}) , then $\#$ would
mean the signed count of isolated points. Similarly, if  one uses a Novikov ring as
the base ring $\Lambda$, then the above differential should be modified
accordingly as usual to accommodate various other quantities of interest
(homotopy class, area,\ldots etc.). The same remark applies to any of the
moduli spaces and corresponding counts that we use in this paper.
\end{remark} 

The compactness and gluing properties of the above moduli spaces allow one to prove that the
differential squares to zero, hence we get a well-defined Floer cohomology group. We
refer the reader to Proposition 5.3.1 in \cite{WW} for a continuation argument which shows that the resulting group
is independent of the choices of $(\delta_i, H_i, J_i)_{i=1,\ldots,k}$.

Following \cite{WWc}, we will prove Theorems \ref{isomorphism},\ref{positive} and
\ref{negative} in a special case
(the general case is proved in exactly the same way).  Let $(M_i,\omega_i)_{i=0,1,2}$ be
symplectic manifolds of dimension $2n_i$ and let $$L_0\subset M_0,\  L_{01}\subset M^-_0\times M_1,\
L_{12}\subset M_1^-\times M_2,\  L_2\subset M^-_2$$ be compact Lagrangian submanifolds
such that the geometric composition $L_{02}=L_{01}\circ L_{12}\subset M_0^-\times M_2$
is embedded. As discussed above, we can perturb $L_0$ and $L_2$ so that
the generalized intersections of $(L_0,L_{01},L_{12},L_2)$ as well as
$(L_0,L_{02},L_2)$ are transverse.  Our goal is to construct a map $$\Phi:
HF(L_0,L_{01},L_{12},L_2) \rightarrow HF(L_0,L_{02},L_2) $$ which we will prove to be
an isomorphism.  Note that there is an obvious bijection of the chain groups
$$CF(L_0,L_{01},L_{12},L_2)\cong CF(L_0,L_{02},L_2)$$  The map $\Phi$ will not
necessarily be induced by this bijection.  As we will later demonstrate, it will differ from this bijection possibly by a nilpotent matrix.
\subsection{Defining the quilt}
\label{construction}

Our construction of $\Phi$ is summarized in Figure \ref{quilt} below.  
\begin{figure}[ht]
\centering
\includegraphics{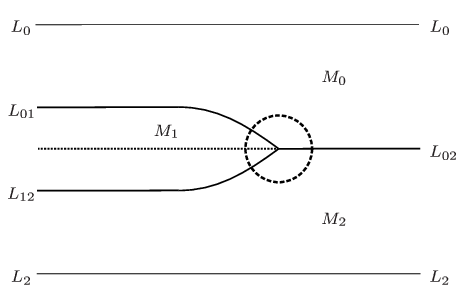}
\caption{The geometric composition map}
\label{quilt}
\end{figure}

Let $\Sigma$ be the pictured {\em quilt} (without the dotted line segment and
the dotted circle). More precisely, let $\Sigma = \Sigma_{0} \cup \Sigma_1 \cup
\Sigma_2 $, where topologically speaking, each $\Sigma_i$ is homeomorphic to a
closed unit disk with punctures at the boundary. Namely,  $\Sigma_0 = D^2
\backslash \{-1,-i,1\}$ , $\Sigma_1 = D^2 \backslash \{-1,1\}$ and $\Sigma_2 =
D^2 \backslash \{-1,i,1\}$. These are the embedded regions (patches) in Figure
$\ref{quilt}$. Ignoring the dotted curves, $\Sigma_i$ is embedded in the region
labelled $M_i$ and $\Sigma$ is obtained by identifying the connected components of the boundary of $\Sigma_i$ as indicated in Figure \ref{quilt}. As pictured, $\Sigma$ has 3 (resp. 2) boundary punctures on the
left (resp. right) which we refer to as incoming (resp. outgoing) ends. There
is also the puncture in the middle, which we will refer to as the Y-end. (One can visualize the Y-end as a semi-infinite cylinder where $3$ semi-infinite strips are glued together along the 3 solid (straight, parallel) lines that go to infinity. On each patch $\Sigma_i$ we fix a complex structure $j_i$ with real analytic boundary conditions as in \cite{WWsurf}. In short, this means that the
seams, the solid curve segments  in Figure \ref{quilt},  are embedded as real
analytic sets in $\Sigma$. A concrete way of arranging such $j_i$ is as
follows: Take a thin neighborhood of each seam in $\Sigma$ and identify it with
a small neighborhood $\f{R} \times (-i\epsilon,i\epsilon)$ of $\f{R}$ in
$\f{C}$ , and choose $\{j_i\} $ to be holomorphically compatible with this
identification. As in Figure \ref{quilt}, let us label the maps from the three
patches as $u_i:\Sigma_i\rightarrow M_i$ with $i=0,1,2$. Along the boundary components of these patches, these have the labeled ``seam conditions'' as in the picture. This means first of all that there are choices of diffeomorphisms $\phi_{ij}$ between adjacent boundary components of each patch (the way we constructed the
$\{j_i\}$ near the seams by using the embedding of $\Sigma_i$ in $\Sigma$
determines these choices). Now, if we consider the maps from adjacent patches,
say $u_i$ and $u_j$, as a map $(u_i, u_j) : \Sigma_i \times \Sigma_j \to M_i
\times M_j$ of product manifolds, for each point $x$ in the boundary of $\Sigma_i$ adjacent to $\Sigma_j$, we should have that $(u_i(x), u_j( \phi_{ij}(x))) \in L_{ij}$ where $L_{ij}$ is the labeled Lagrangian submanifold in the product $M_i \times M_j$. What we describe here is a particular example of a holomorphic quilt. We refer to Definition 3.1
in \cite{WWsurf} for a more detailed and general definition of a holomorphic
quilt. 

We may identify the region inside the dotted circle, which is a neighborhood of the Y-end, with
$(0,\infty)\times [0,1]$ mapping to $M^-_0\times M_1\times M^-_1\times M_2$.
More precisely, first we split the strip corresponding to $u_1:\Sigma_1\to M_1$
along the dotted horizontal seam in Figure \ref{quilt} and put the diagonal
seam condition $\Delta \subset M_1\times M^-_1$. This has no effect on the
moduli space that we consider. However, now at the Y-end we can ``fold'' the
strip to get the desired map.  Specifically, at the Y-end instead of looking at
maps from different strips to different manifolds, one can consider a single
map from $(0,\infty) \times [0,1]$ to the product $M^-_0\times M_1\times
M^-_1\times M_2$.  Therefore, we choose our complex structure $j_i$ so that
near the Y-end they are identified with the standard complex structure on
$(0,\infty) \times [0,1]$.  Similarly, we can choose $j_i$ near the incoming
and outgoing ends so that we can identify our strips with $(-\infty,0]\times
[0,1]$ and $[0,\infty)\times [0,1]$ and we fix these choices once and for all. (For a standard discussion about these choices see the discussion of strip-like ends in \cite{Seidel} Section (8d)). At the Y-end, let us label the map
obtained by folding by $$v:(0,\infty)\times [0,1] \rightarrow M^-_0\times
M_1\times M^-_1\times M_2=\underbar{M}$$ This has the seam conditions
$v(s,0)\in L_{01}\times L_{12}$ and $ v(s,1)\in L_{02}\times \Delta$. (Note that at this point we do not impose any condition on the behavior of $v$ as $s \to \infty$. It will be seen to converge to an intersection point of $L_{01} \times L_{12}$ and $L_{02} \times \Delta$ as a consequence of holomorphicity.) Next, we
would like to specify the complex structures on each $M_i$.  Assume that we
have chosen $J_i$ on $M_i$ such that $HF(L_0,L_{01},L_{12},L_2)$ and
$HF(L_0,L_{02},L_2)$ are both defined.  In general, such $J_i$ may need to be
$t$-dependent near $$(s,t)\in [0,\infty)\times [0,1]\cup (-\infty,0]\times
[0,1]$$ to ensure transversality for the moduli spaces that appear in the
definition of the Floer differential.  Note that this specifies
$\underbar{J}=J_0\times -J_1\times J_1\times -J_2$ on $M^-_0\times M_1 \times
M^-_1\times M_2$.  To ensure transversality for the moduli space of quilted
maps, we now introduce a domain dependent $J(z)$ on $M$. Pick a small
holomorphically embedded disk $D\subset (0,\infty)\times (0,1)$ (Note that this
is an interior disk).  We define $J(z)$ by letting $J(z)=\underbar{J}$ outside
$D$ and letting $J(z)$ be chosen generically from the set of compatible complex
structures inside $D$.  Such a $J(z)$ need not preserve the product structure
on $\underbar{M}$.  A similar construction in quilted Floer theory already
appears in \cite{PerutzG}.  We will denote such domain dependent complex
structures for our quilt, simply by $J$.

\begin{definition} \label{defin} Let $\underbar{x},\underbar{y}$ be two generalized intersection points for
$(L_0,L_{02},L_2)$ (or equivalently $(L_0,L_{01},L_{12},L_2)$). Let
$\mathcal{M}_{J}(\underbar{x},\underbar{y})$ be the set of all finite energy maps
$\underbar{u}=(u_i)_{i=0}^2$ that are
holomorphic with respect to $J$, have the quilted Lagrangian boundary conditions and
converge to $\underbar{x}$ on the incoming end and to $\underbar{y}$ on the outgoing end.  
\end{definition}       

Note that $L_{02}\times \Delta$ and $L_{01} \times L_{12}$ intersect cleanly in
$$\tilde{L}_{02}= (L_{02}\times \Delta) \cap (L_{01}\times L_{12})$$ which is
diffeomorphic to $L_{02}$. By definition, this means that $$T\tilde{L}_{02} = T(L_{02}\times \Delta)
\cap T(L_{01}\times L_{12})$$ The finite energy assumption guarantees
that the map near the Y-end has exponential decay.  More precisely, at the Y-end we have a holomorphic map
$$v:(0,\infty)\times [0,1]\rightarrow M^-_0\times M_1\times M^-_1\times M_2$$
for which we have the following decay estimate (see \cite{WoodG}, lemma 2.5 or \cite{IvShev}, appendix 3):
\begin{lemma}
There exists $\ep_0>0$, such that for any holomorphic $v$ with finite energy there exists $C$ such that $$sup_{t\in [0,1]}|\nabla v(s,t)|\leq Ce^{-\ep_0 s}$$ 
\end{lemma}

Here the gradient is taken with respect to a reference metric and the constant
$C$ only depends on this metric. This lemma combined with Gromov-Floer compactness
implies that as $s\rightarrow \infty$ each $v$ converges exponentially fast to
some point $z\in \tilde{L}_{02}$ and all derivatives of $v(s,t)$ converge to
$0$ exponentially fast. We will denote the point of convergence by $v(\infty)$.

Note that the use of Gromov-Floer compactness is justified by the following
standard argument:  We cover our quilt by a finite number of domains such that
each domain can be folded to yield a map of a holomorphic curve from a single
Riemann surface (possibly with boundary and strip-like ends) into a product of
symplectic manifolds. (Note that the folding can be carried out even when there
are boundary conditions as one requires seams to be embedded as real-analytic
subsets). On each domain we apply the usual Gromov-Floer compactness to deduce
convergence in compact subdomains outside a finite collection of points. This
implies Gromov-Floer compactness holds in any compact subdomain of a quilt.
Finally, at the strip-like ends, we can again fold and reduce to the standard
Gromov-Floer compactness statement. Putting this all together yields the
desired Gromov-Floer compactness statement for a quilt. (See Chapter 4 of
\cite{MS} for a detailed discussion of compactness in the case the domain is compact which applies to the setting of holomorphic quilts without any change and Theorem 2.14 of \cite{WoodG} concerning strips with clean intersection Lagrangian boundary conditions).

\subsection{Morse-Bott intersections and transversality}
Given that any element $u\in
\mathcal{M}_{J}(\underbar{x},\underbar{y})$ has exponential decay for
some uniform $\ep_0>0$ at the Y-end, we can view  $
\mathcal{M}_{J}(\underbar{x},\underbar{y})$ as the zero set of a
Fredholm section of a Banach bundle defined using a norm with exponential weights at the Y-end. We briefly review this
construction with the purpose of identifying the relevant tangent spaces. \\\\ 
 Fix some $p>2$, intersection points $x\in CF(L_0,L_{01},L_{12},L_2)$ and $y\in
CF(L_0,L_{02},L_2)$.  Let $(0,\infty)\times [0,1]$ be a neighborhood of the Y-end in
the
quilt $\Sigma$.  Let $\underline{\Sigma}=\Sigma-(1,\infty)\times [0,1]$ be the
complement of a slightly smaller end.  On the open subdomain $\underline{\Sigma}$, we define the
Banach manifold $B_1$ of all $L^p_{1,\ep}$ maps with quilted boundary
conditions (seam conditions) that converge to $\underbar{x}$ on the incoming end and to $\underbar{y}$ on the
outgoing end. For any sufficiently small $\ep>0$, we may define on $(0,\infty)\times
[0,1]$ the Banach manifold
$B_2$ of all $L^p_{1,\ep}$ maps $$v:(0,\infty)\times [0,1]\rightarrow
M_0\times M_1\times M_1\times M_2=\underbar{M}$$
with Lagrangian boundary conditions $v(s,0)\in L_{01}\times L_{12}$, $ v(s,1)\in
L_{02}\times \Delta$ and exponential decay with coefficient $\ep$. For a recent review
of exponential weights
(with references to older treatments) see \cite{WoodG}. Any element $v\in B_2$
converges to some point on the manifold $ (L_{01}\times L_{12}) \cap
(L_{02}\times \Delta)$.  The tangent space to such $v$ are a pair
$$v'=(v'_a,v'_b)$$ where $v'_a$ is a section of $v^*(T\underbar{M})$ with
totally real boundary conditions and exponential decay at infinity, while $v'_b$ is an element of the finite
dimensional space
$$T_{v(\infty)}( ( L_{01}\times L_{12}) \cap  (L_{02}\times \Delta))$$ The norm on $v'$ is specified by   
 $$\left( \int_{(0,\infty)\times [0,1]}|
e^{\ep s}v'_a|^p+|\nabla (e^{\ep s}v'_a)|^p dsdt \right)^{1/p}+|v'_b|$$
Here we may use any norm $| . |$ on the tangent spaces induced by  smooth Riemannian metrics on the $M_i$'s.  Since we assume that each $M_i$ is compact, all such metrics lead to equivalent topologies on the tangent space to $v$. \\\\ 
A chart of $B_2$ near $v$ is obtained by applying the exponential map to all such $v'$ where the
norm of  $v'$ is sufficiently small. To be precise, one must use a
$t$-dependent metric on $\underbar{M}$ which makes $L_{01}\times L_{12}$
totally geodesic for $t=0$ and $L_{02}\times \Delta$ totally geodesic for
$t=1$. See \cite{WoodG} Section 2.2) for more details.   Finally, we define our Banach
manifold $\mathcal{B}_{\epsilon}(\underbar{x},\underbar{y})$ over $\Sigma$ as
pairs $(w,v)\subset (B_1,B_2)$ which agree on the overlap $(0,1)\times [0,1]$.\

Now, let $\mathcal{V}$ be the Banach bundle over
$\mathcal{B}_{\epsilon}(\underbar{x},\underbar{y})$ whose fibre over
$\underbar{u}$ is given by $\Omega^{0,1}(\Sigma, \underbar{E})$ where
$\underbar{E}$ is the pullback of the tangent bundles of $M_i$ and
$\Omega^{0,1}(\Sigma, \underbar{E})$ denotes the space of $(0,1)$-forms with
finite $L^p$ norm and with exponential decay at the Y-end.  On each of the two
pieces, standard arguments (see \cite{WW} for $B_1$ and \cite{WoodG} for
$B_2$) imply that the $\delbar$ operator is a restriction of a Fredholm
operator to an open domain. Note that this does \emph{not} mean that
restriction of $\delbar$ to $B_1$ or $B_2$ is a Fredholm operator; we are 
saying that there exists a way to embed the open domains $\underline{\Sigma}$
and $(1,\infty) \times (0,1)$ to bigger domains where one has extensions of
the corresponding $\delbar$ operators to Fredholm operators. One can arrange
this easily in the case of $B_1$ and $B_2$. For example, one can take the
double of each domain, and at the same time doubling all the structure of Banach bundles and the sections defined by $\delbar$ operators (see \cite{HLS} for a discussion of such doubling in a related problem). In
that way, $B_1$ is embedded into a standard Fredholm problem for a holomorphic
quilt as studied in (\cite{WW}, pg. 877) and $B_2$ is embedded into a standard
Fredholm problem that underlies the definition of Floer differential for the
Lagrangian Floer homology of $(L_{01} \times L_{12}, L_{02} \times \Delta)$ (\cite{WoodG}, pg. 14).\\\\
Knowing that the $\delbar$ operator is a restriction of a Fredholm operator to
an open domain for both $B_1$ and $B_2$ allows us to apply a standard patching
argument (see for example \cite{Donald} pg. 50) which implies that the $\delbar$
operator defines a Fredholm section over
$\mathcal{B}_{\epsilon}(\underbar{x},\underbar{y})$.  Note that $\ep$ has to be
chosen sufficiently small to ensure that $\delbar$ on $B_2$ is the restriction of a
Fredholm operator to an open domain and that each element of
$\mathcal{M}_{J}(\underbar{x},\underbar{y})$ actually belongs to
$\mathcal{B}_{\epsilon}(\underbar{x},\underbar{y})$.\\\\ For future reference,
note that the linearization of $\delbar$  at some $v$ on $(0,\infty)\times
[0,1]$ has the form: $$D\delbar\oplus K: L^p_{1,\ep}(\underbar{E},
\underbar{F}) \oplus T_{v(\infty)}( ( L_{01}\times L_{12}) \cap (L_{02}\times
\Delta))\rightarrow L^p_{0,\ep}(\Omega^{0,1}(\underbar{E}))$$ 
Here $\underbar{E} = v^* T\underbar{M}$ and $\underbar{F}$ is the Lagrangian subbundle given by $v^* (T(L_{01} \times L_{12}))$ along $\{0, \infty\} \times \{0\}$ and $v^* (T(L_{02} \times \Delta))$ along $\{0, \infty\} \times \{ 1 \}$; $K$ is some operator with a finite dimensional
linear domain $T_{v(\infty)}( ( L_{01}\times L_{12}) \cap  (L_{02}\times \Delta))$.  The
specific form of $K$ depends on the choice of the $t$-dependent metric $g_t$ and will not need to be made explicit for our purposes.  The special case
when $v$ is constant will be discussed below (Section 2.4). We now give a proof of the
following claim:

\begin{proposition} \label{transverse} For a generic choice of $J$,
$\mathcal{M}_{J}(\underbar{x},\underbar{y})$ is a smooth finite
dimensional manifold.
\end{proposition}

\begin{proof}
We need to verify that,
for a generic choice of $J$, the Fredholm section defined by $\delbar$ of the Banach bundle $\mathcal{V}$ over $B_\epsilon(\underbar{x}, \underbar{y})$ will be transverse to the
zero section. Let $\mathcal{J}$ be the space of almost complex structures
constructed above (see page 8). Recall that we consider domain dependent almost complex structures where
this dependence is at the Y-end and only on a small disk $D \subset (0,\infty) \times
(0,1)$.  Thus, given that $u_i\in \mathcal{M}_{J}(x,y)$, we need to
show that the
linearized operator $D\delbar(\underbar{u},J)$ is surjective.  Note
that as we use a
domain dependent $J$,  the  linearized operator is the sum of two terms - one term coming from
$D\delbar(\underbar{u})$ corresponding to variations of $\underbar{u}$
and the second term corresponds to variations of $J$. First assume that the map $v$ defined at the
Y-end is
non-constant. We will
show that any $L^q_{0, -\epsilon}$-section ($1/p+1/q=1$) $\eta$ of the dual bundle $\Omega^{0,1}(\underline{E}, \underline{F}) $ orthogonal to the image of the linearization must vanish around some
point in $D$.  Then the unique continuation principle will yield that
$\eta$ vanishes
identically.

We now write the linearization of our section on the disk $D \subset (0,\infty) \times
(0,1)$, where we required $J$ to have the domain dependence.   The main point here is
that since we allow our $J$ to be domain dependent, we do not need a somewhere
injective curve but simply a point $z_0\in D$ such that $dv(z_0)\neq 0$. Following the
argument in \cite[page 48]{MS}, the linearized operator has the following form on $D$:
$$D\delbar(v,J)=  D\delbar(v) + \frac{1}{2} S(v,z) \circ dv \circ j_{D}$$ Here,
$D\delbar(v)$
denotes the partial derivative holding $J$ fixed and $S(z,v) \circ{v} \circ j_{D}$
corresponds to linearization with respect to $J$, where $S(z,v)$ is a section of the
tangent space to $\mathcal{J}$ at $J$ which can be identified with the subspace of $S \in End(TM)$ such that $$SJ=-JS, \ \ \omega(S\cdot , \cdot)=-\omega(\cdot,S\cdot )$$  Now, suppose that
some section $\eta$ is orthogonal to the image.  Following \cite{MS}, we can choose
$S(z_0, v(z_0))$ such that $$\langle \eta(z_0), S(z_0, v(z_0)) \circ dv(z_0)\circ j_D
\rangle >0$$ whenever $\eta(z_0) \neq 0 $.  We can extend $S(z_0,v(z_0))$ to a small
neighborhood of $z_0$ by using a bump function.  Note that the resulting $S(z,v)$ is
domain dependent.  This shows that $\eta(z_0)=0$ for all $z_0$ where $dv(z_0)\neq 0$.
However, such $z_0$ are dense in $D$.    Since $\eta$ is in the kernel of a
$\delbar$-type operator, namely $(D\delbar(v))^*\eta =0$, it must vanish everywhere by
the unique continuation principle.

Finally, assume that $v$ is constant, thus by unique continuation
$\underbar{u}$ is constant. In the next subsection, we show that the index of
the linearization is 0 (Lemma \ref{index}) for the case of a constant almost
complex structure of split-type, that is, it comes from an almost complex
structure $J_i$ on $M_i$. Notice however that in Section \ref{construction}
(Definition \ref{defin}), we have required that almost complex structure $J$ is
of split-type except on a small disk $D$ in order to ensure transversality at
the constant solutions. On the other hand, for the purpose of calculation of
index we can always homotope our $J$ to a split-type almost complex structure
so the calculation of the next subsection ensures that the index of the
linearization is 0 for a regular $J$ (as in Definition \ref{defin}). 

We claim that for any regular $J$, $D\delbar(\underbar{u},J)$ is surjective.
In view of the index calculation, it is enough to show that the kernel of the
linearization at a constant map is zero. We may identify the image of $u_i$
with $0\in \Ar^{2n_i}$. An element $\underline{u}'$ of the kernel of the linearization is then a
triple of maps $u_i'$ from the quilt to $(\Ar^{2n_i},\omega_0)$. (Outside of the region where $J$ is non-split, we have that $\delbar_{J_i} u'_i=0$ but along the disk $D$ where $J$ is non-split we can only speak of $J$-holomorphicity for the folded map.) $u_i'$ have linear quilted Lagrangian boundary
conditions.  Thus, along each seam $(u_i',u_{i+1}')\in L'_{i,i+1}\subset
\Ar^{2(n_i+n_{i+1})}$ where $L'_{i,i+1}$ is a linear Lagrangian submanifold. 
By construction,
$u_i'$ have exponential decay near the incoming and outgoing ends. Near the Y-end, we can identify $u_i'$ with the folded map
$$v':(0,\infty)\times [0,1]\rightarrow \C^{n_0+2n_1+n_2}$$ which satisfies $\delbar_{J}v'=0$. Let us decompose
$v'$ as $$v'=v_a'+v_b'$$  where $v_a'$ has exponential decay at infinity and
$v_b'$ is a constant vector in the intersection of the linear Lagrangians. In particular, the total derivative of $v'$ has exponential decay near
infinity.  Note that in the part where $J$ does not split (i.e. near the Y-end)
by assumption $J$  is compatible with the symplectic structure on
$$(\Ar^{2n_0},-\omega_0)\times (\Ar^{2n_1},\omega_0) \times
(\Ar^{2n_1},-\omega_0) \times (\Ar^{2n_2},\omega_0)$$  

We claim that in fact each $u_i'$ is constant. Now, the standard
symplectic form $\omega_0 = \frac{1}{2} d(xdy-ydx)$ is an exact form and the
linear Lagrangians $L'_{i,i+1}$ are exact Lagrangians, therefore by Stokes'
theorem we have that the symplectic area vanishes: $$\sum_{i=1}^3 \int_{\Sigma_i}
(u_i')^{*}\omega_0 =0$$ Note that the use of Stokes' theorem is justified
since near the incoming and outgoing ends $u_i'$ have exponential decay and
near the Y-end we have that the total derivative of $v'$ decays exponentially.
Now, the fact that $\underline{u}'$ is holomorphic enables us to relate the symplectic area to the energy (see \cite[page 20]{MS}). Since we have that $J$ is non-split near the Y-end, in order express the energy, as before let us divide the domain into two pieces. Near the Y-end, we have the folded map $v'$ with domain $(0,\infty) \times [0,1]$. Denote by $\Sigma'_i = \Sigma_i \backslash  ( (0,\infty) \times [0,1] )$, the domain of $u'_i$ outside of the neigborhood of the Y-end. We can now write out the relation between symplectic area and the energy of a holomorphic $\underline{u}'$ as follows:
$$\sum_{i=1}^3 \frac{1}{2}\int_{\Sigma'_i} |du'_i|^2_{J_i}dvol_{\Sigma'_i} + \frac{1}{2} \int_{(0,\infty) \times [0,1]} |dv'|^2_{J} dvol  = \sum_{i=1}^3 \int_{\Sigma_i} (u'_i)^{*}\omega_0 = 0$$
where $J_i$ are almost complex structures on $M_i$ and $J = (J_i)_{i=1}^3$ over $(\Sigma'_i)_{i=1}^3$. 

Therefore, we conclude that the $u_i'$ are constant.  Since $u_i'$
converges to zero at the incoming and outgoing ends, this implies that $u_i'=v'=0$ as
desired.
\end{proof}

\subsection{Index computation for a constant map} 
\label{Mindex}
In this section we calculate the index of the linearization of $\delbar$ operator at a constant map. Although, everywhere else, we used Sobolev spaces $L^p_{1,\epsilon}$ and $L^p_{0,\epsilon}$ for $p>2$, in calculating the index we take $p=2$ as this makes the calculation of the index easier. Since the kernel and cokernel of the linearized $\delbar$ operator is smooth by elliptic regularity, this calculation immediately implies the index calculation for $p>2$.

\begin{lemma} 
	\label{index}
	Let $u \in \mathcal{M}_J(\underline{x},\underline{x})$ be a constant map, then the index of $u$ vanishes.
\end{lemma}

The linearized problem we will study is that of holomorphic quilts mapping into
$\mathbb{C}^n$ with linear Lagrangian boundary conditions.  As preparation for
the main result, we first review a standard Morse-Bott index calculation.  Let
$L$, $L' \subset \mathbb{C}^n$  be a pair of Lagrangian subspaces.  Let
$S=\Ar\times [0,1]$.  We consider the Fredholm map $$\delbar:
L^2_{1,\ep}(S;L,L') \rightarrow L^2_{0,\ep}(S)$$ for sufficiently small
$\ep>0$.  Here $L^2_{1,\ep}(S;L,L')$ denotes the weighted Sobolev space of
maps with $u(\cdot,0)\in L$ and $u(\cdot,1)\in L'$. Note that in view of the
restriction map $$L^2_{1}(S)\rightarrow L^2(\partial S )$$ we do not
need $u$ to be continuous to make sense of the linear Lagrangian boundary
condition. For a general discussion of regularity for the $\delbar$-operator
with totally real boundary conditions that include the rather special case we
are considering, see Theorem C.1.10 in \cite{MS}.  
\begin{lemma}\label{lemmaMB}
	$ind(\delbar)=-\text{dim}(L\cap L')$.     
\end{lemma} 
\begin{proof} A
		function $u:\Ar\times [0,1] \to (\mathbb{C}^n; L, L')$ may be
		written as $$u(s,t)=\sum_\lambda f(t)\phi_\lambda(s)$$ where
		$\phi_\lambda$ is an eigenfunction with eigenvalue $\lambda$ of
		the operator $-i\partial_s$ on $[0,1]$ with $\phi_\lambda(0)\in
		L$ and $\phi_\lambda(1)\in L'$. \\\\ The kernel of
		$\delbar=\partial_t+i\partial_s$ consists of maps
		$$u(s,t)=\sum_\lambda c_\lambda e^{t\lambda}\phi_\lambda(s)$$
		for some constants $c_\lambda$. However, since $\lambda$ is
		real and such solutions are required to have exponential decay
		for $t\rightarrow \pm \infty$, it must be that $c_\lambda =0$
		for all $\lambda$. Hence, we conclude that
		$\text{dim}\ker\delbar=0$.

By elliptic regularity, the cokernel can be identified with the kernel of $-\partial_t+i\partial_s$ on the
space of $L^2_1$ functions with exponential growth of at most $\ep$.  In addition,
these functions have boundary values on $iL$ and $iL'$.   Therefore, such maps consist
of $$u(s,t)=\sum_\lambda c_\lambda e^{-t\lambda}\phi_\lambda(s)$$ However, if $\ep$ is
smaller than the first nonzero eigenvalue $\lambda_0$, the only maps are those with
$\lambda=0$.  These are precisely the constant maps with values in $(iL)\cap (iL')$.
Therefore,  index $\delbar=-\text{dim}((iL)\cap (iL'))=-\text{dim}(L\cap L')$.
\end{proof} 
The calculation of the index for $$\delbar: L^2_{1,\ep}(S;L,L')
\rightarrow L^2_{0,\ep}(S)$$  computes the index for any $$\delbar: L^p_{k,\ep}(S;L,L')
\rightarrow L^p_{k-1,\ep}(S)$$ with $p\geq 2$ and $k \geq 1$.  Indeed, elliptic regularity for the $\delbar$-operator with linear Lagrangian boundary conditions implies that any element of the kernel/cokernel is smooth.  Strictly speaking, one must first prove that the corresponding problem for $L^p_{k,\ep}$-spaces is Fredholm.  As explained in \cite{Donald} (pg.58-60, 70-75) , one can convert the index problem over weighted spaces to an equivalent problem for unweighted spaces where the Fredholm property is standard.
The previous lemma is useful when considering Morse-Bott moduli spaces.  In particular,
consider the tangent space at the constant map of the moduli space of holomorphic
curves with Morse-Bott boundary conditions along $(L,L')$.  By definition, it is the
kernel of the map $$\delbar \oplus K: L^2_{1,\ep}(S;L,L')\oplus ( (L\cap L')\times
(L\cap L') ) \rightarrow L^2_{0,\ep}(S)$$  For the calculation of the index the
explicit form of the map $K$ is not relevant since it is a compact operator. Thus we have:
\begin{corollary}
$ind(\delbar \oplus K)= \text{dim}(L\cap L')$.
\end{corollary}

This is consistent with the
intuition that the Morse-Bott  case corresponds to constant holomorphic disks lying on
$L\cap L'$.\\\\ We will make use of excision for our index calculations. This is a
standard tool for computing the index of elliptic operators that goes back to the work
of Atiyah and Singer on the index theorem (\cite{AS}).  We review a simple version of it that is
tailored to our application.  For recent proofs, one may consult \cite{Ben}.  

\begin{figure}[ht]
\centering
\includegraphics[scale=0.9]{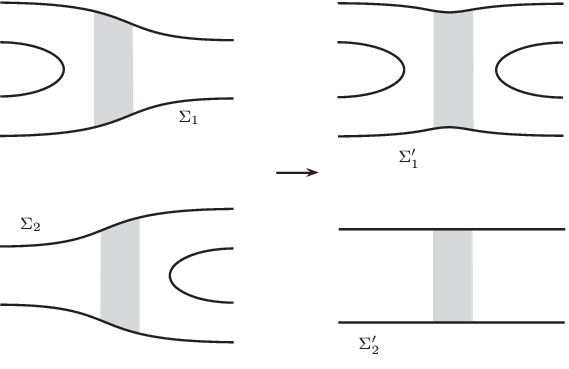}
\caption{Excision}
\label{excisionX}	
\end{figure}

Suppose we are given quilts $\Sigma_1$, $\Sigma_2$ each with a pair of complex vector bundles $E_i$ and $F_i$.  In addition, suppose we have $\delbar$-operators $$\delbar_i:\Gamma(E_i)\rightarrow \Gamma(F_i)$$ over each $\Sigma_i$.  At the boundaries, we assume that there are totally real boundary conditions.  This amounts
to a choice of a totally real subbundle $T_i$ of each $E_i$ over the boundary of
$\Sigma_i$.\\\\  Now, assume that each $\Sigma_i$ contains a separating strip
$(a,b)\times [0,1]$.   We assume there are isomorphisms  $F:E_{1|(a,b)\times
[0,1]}\rightarrow E_{2|(a,b)\times [0,1]}$ and $G:F_{1|(a,b)\times [0,1]}\rightarrow
F_{2|(a,b)\times [0,1]}$ which map $\delbar_1$ to $\delbar_2$ and $T_1$ to $T_2$, respectively.  We
may excise $\Sigma_i$ along the strips as in Figure \ref{excisionX} to form new quilts $\Sigma_1'$
and  $\Sigma_2'$ with corresponding bundles and $\delbar$-operators $\delbar_1'$ and
$\delbar_2'$. The excision theorem asserts that
$$\text{ind}(\delbar_1)+\text{ind}(\delbar_2)=\text{ind}(\delbar_1')+\text{ind}(\delbar_2')$$
A similar discussion applies when instead of a separating strip we have a separating
cylinder $(a,b)\times S^1$

\begin{figure}
\centering
\subfloat{Fig 1}{\label{fig:1}\includegraphics[width=0.35\textwidth]{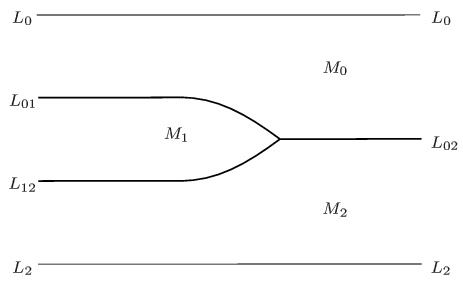}}                
\hspace{1cm} 
\subfloat{Fig 2}{\label{fig:2}\includegraphics[width=0.35\textwidth]{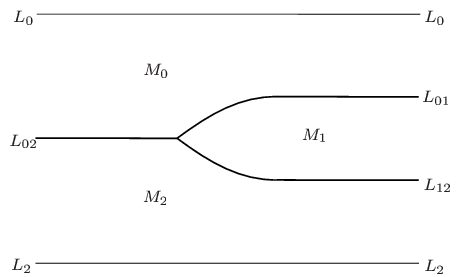}}
\\
 \subfloat{Fig 3}{\label{fig:3}\includegraphics[width=0.35\textwidth]{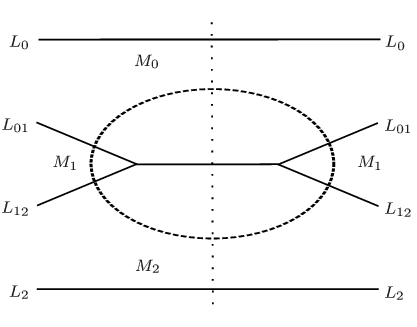}}
\hspace{1cm}
 \subfloat{Fig 4}{\label{fig:4}\includegraphics[width=0.35\textwidth]{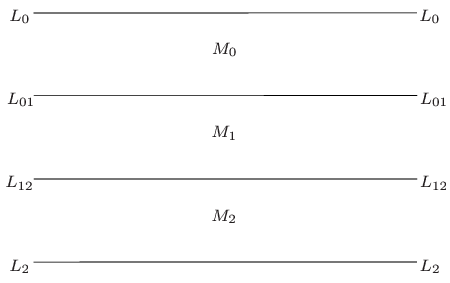}}                
\\
\subfloat{Fig 5}{\label{fig:5}\includegraphics{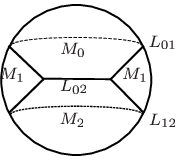}}
\hspace{1cm}
 \subfloat{Fig 6}{\label{fig:6}\includegraphics{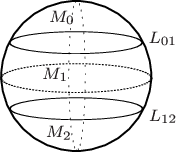}} 
\\
\vspace{0.5cm}
 \subfloat{Fig 7}{\label{fig:7}\includegraphics{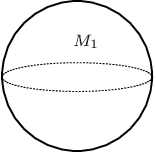}}                
 \subfloat{Fig 8}{\label{fig:8}\includegraphics{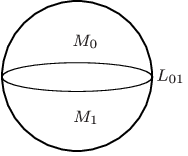}}
 \subfloat{Fig 9}{\label{fig:9}\includegraphics{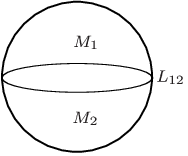}}
\caption{Index calculation using excision}
 \label{excision}
\end{figure}

We are ready to compute the index of the linearization at a constant map.  Note
that for this linearization all maps are into $\f{C}^n$ with the standard complex
structure and the nonlinear Lagrangian boundary conditions are replaced by their
tangent spaces in $\f{C}^n$.    Consider the nine figures drawn in Figure
\ref{excision}.  Let $m_i$ stand for the index of Fig $i$.  We wish to compute $m_1$.
We have shown in the previous section that the kernel of the map represented by Fig 1 is
zero.  Similarly the kernel of Fig 2 is zero.  This implies that $m_1\leq 0$ and
$m_2\leq 0$.  By additivity of index,  $$m_3=m_1+m_2$$  Excising Fig 3 and 6 along a neighborhood of the dotted
circles (where we use the vertical dotted circle for Fig 6) gives $$m_3+m_6=m_4+m_5$$  Now, we claim that $m_5=\text{dim}(L_{02})$.  To see
this, one simply folds to obtain a single strip with Morse-Bott Lagrangian boundary
conditions on $(L_{01}\times L_{12} , L_{02} \times \Delta)$.  Thus, the discussion
right after Lemma \ref{lemmaMB} above
gives $m_5=\text{dim}(L_{02})$.  We have $m_4=0$ since it is the quilt of the identity map.  To compute $m_6$, note
that excision (this time we use horizontal dotted circle for Fig 6) implies that $$m_6+m_7=m_8+m_9$$  By folding, we have that $m_8$ and
$m_9$ represent disks so  $$m_8+m_9=\text{dim}(L_{01})+\text{dim}(L_{12})$$ and $m_7= \text{dim}(M_1)$ since it is the linearization of a constant map of a sphere. Thus, $m_6=\text{dim}(L_{02})$ which
together with $m_4=0$ and $m_5= \text{dim}(L_{02})$ gives $m_3=0$. This implies $m_1=m_2=0$, as desired.

\subsection{Completion of the proof of Theorem \ref{isomorphism}} 

By Proposition \ref{transverse}, the moduli spaces $\mathcal{M}_J(\underbar{x},\underbar{y})$ are transversely cut out. To define a count we need to show that the zero dimensional moduli spaces
$\mathcal{M}^0_J(\underbar{x},\underbar{y})$ is compact and hence finite. Then,
$\mathcal{M}^0_J(\underbar{x},\underbar{y})$ allows us to define the map $$\Phi:
CF(L_0,L_{01},L_{12},L_2)\rightarrow CF(L_0,L_{02},L_2)$$ We will sometimes refer to this map by Y-map. To verify that this is
indeed a chain map we need to consider the 1-dimensional moduli spaces
$\mathcal{M}^1_J(\underbar{x},\underbar{y})$. 

First note that the set of assumption (1) on second homotopy classes ensure that we
cannot have any interior disk or sphere bubbles. Therefore, by Gromov compactness the
boundary of $\mathcal{M}^0_J(\underbar{x},\underbar{y})$ and
$\mathcal{M}^1_J(\underbar{x},\underbar{y})$ consists of broken configurations at the
ends. In the case of $\mathcal{M}^0_J(\underbar{x},\underbar{y})$, there cannot be
breaking at the $\underbar{x}$ and $\underbar{y}$ ends because by our transversality
assumptions such a break cannot occur in a 0-dimensional component of the moduli space.  Finally, we need to
argue that for both $\mathcal{M}^0_J(\underbar{x},\underbar{y})$ and
$\mathcal{M}^1_J(\underbar{x},\underbar{y})$ there cannot be a breaking at the Y-end. 

To this end, the following lemma will be useful:

\begin{lemma}
	\label{disk}
	Let $\delta : \f{R} \times [0,1] \to (\underbar{M} ; L_{01} \times L_{12} , L_{02} \times \Delta ) $ be a smooth map with Lagrangian boundary conditions and at the two ends converges exponentially to points in the Morse-Bott intersection. Then there exists a smooth map $\tilde{\delta} \in \pi_2(\underbar{M}; L_{01} \times L_{12})$ such that \[ \int \delta^*\omega_M = \int \tilde{\delta}^*\omega_M \]
	Furthermore, the Fredholm indices of $\delta$ and $\tilde{\delta}$ are related as follows: \[ \text{index}(\tilde{\delta}) + 2\text{dim}(M_1) = \text{index}(\delta) \]
\end{lemma} 

\begin{proof} It will be convenient to view $\delta$ as a quilted map $$\delta_i:\Ar \times [0,1]\rightarrow M_i, i=1,2,3$$ with cyclic Lagrangian boundary conditions $(L_{01},L_{12},L_{02})$.  Thus, we have $(\delta_2(1,s),\delta_0(0,s))\in L_{02}$, etc. Let $\delta_4:\Ar\times
[0,1]\rightarrow M_1$ be the map with $\delta_4(s,t)=b(s)$, where $b(s)$ is the unique
point on $M_1$ with $(\delta_2(1,s),b(s),b(s),\delta_0(0,s))\subset L_{01}\times
L_{12}$. Note that $\delta_4$ is a smooth map which is not holomorphic but converges
exponentially as $|s|\rightarrow \infty$.  Furthermore, the image of $\delta_4$ is
just a path, thus $\delta_4$ has zero area.  We have now obtained a new quilt
$\delta'$ with four patches $\delta_i$ and seams $(L_{01},L_{12},L_{01},L_{12})$ while the area of $\delta$ and $\delta'$ is the same.  We fold $\delta'$ to obtain a map  $$\Ar\times
[0,1] \rightarrow \underbar{M}$$ with boundary on $(L_{01}\times L_{12},L_{01}\times L_{12})$.  Alternatively, we may view this as a map $$\tilde{\delta}:D\rightarrow
\underbar{M}$$ where $D$ is the unit disk and $\tilde{\delta}$ has Lagrangian boundary conditions on $L_{01}\times L_{12}$. This map satisfies the required property, since $\delta,\delta', \tilde{\delta}$ all have the same area.

To see the relation of Fredholm indices, we note that $\delta, \delta', \tilde{\delta}$ all have the same Maslov index since $t$-derivative of $\delta_4$ vanishes (cf. \cite{WWc} pg. 846). The Fredholm index is given by the sum of Maslov index and the dimension of the Morse-Bott intersection. To conclude, observe that the dimension of the Morse-Bott intersection for $\delta$ is $\text{dim}((L_{01} \times L_{12}) \cap (L_{02} \times \Delta)) = \text{dim}(M_0) + \text{dim}(M_2)$ while it is $\text{dim}(L_{01} \times L_{12}) = \text{dim}(M_0) + 2 \text{dim} (M_1) + \text{dim}(M_2)$ for $\tilde{\delta}$. \end{proof} 

Back to the proof of Theorem \ref{isomorphism}, observe that a bubble at the Y-end would be a holomorphic map $\delta:
\f{R} \times [0,1] \to M_0^- \times M_1 \times M^-_1 \times M_2$ as in the previous lemma. However, holomorphicity ensures that it has positive area. Therefore, we would obtain an element $\tilde{\delta} \in \pi_2(\underbar{M}; L_{01} \times L_{12})$ which has positive area, which is impossible by the assumption (1). Thus, there cannot be a bubbling at the Y-end.

Therefore, standard gluing theory applied to
$\mathcal{M}^1_J(\underbar{x},\underbar{y})$ shows that $\Phi$ is a chain map. 

\begin{remark} Note
that we do not need to consider Morse-Bott gluing as the only place where a breaking could occur is at the
ends where we have transverse intersection.   
\end{remark}

To complete the proof we need to show that $\Phi$ induces an isomorphism on cohomology. Let us write $\mathcal{P}^{in} = \{ (\gamma_{0}, \gamma_{1}, \gamma_{2}) |\ \gamma_i :
[0,1] \to M_i, \gamma_{0}(0) \in L_0, \ (\gamma_{0}(1), \gamma_{1}(0)) \in L_{01}, \
(\gamma_{1}(1), \gamma_{2}(0)) \in L_{12}, \gamma_{2}(1) \in L_2   \} $. Each generator of the chain complex $CF(L_0,L_{01},L_{12},L_2)$ is an element of $\mathcal{P}^{in}$ and the chain complex splits into a direct sum of chain complexes corresponding to the path components of $\mathcal{P}^{in}$. In what follows, we assume that $\mathcal{P}^{in}$ is path-connected in order to avoid the notational complexity of indexing path components and carrying out the argument for each path-component separately. 

As above, assumption (1) enable us to have a well-defined action functional, 
\begin{eqnarray*} 
\mathcal{A}^{in} : \mathcal{P}^{in} \to \f{R}  \\ 
\gamma \to \sum_{i=0}^2 \int_{[0,1]^2} (\gamma_i^t)^* \omega_i
\end{eqnarray*}

where as before $\gamma_i^t$ is any choice of a smooth homotopy in $\mathcal{P}^{in}$
between $\gamma_i$ and a fixed path on $\mathcal{P}^{in}$, and $\omega_i$ are the given symplectic
forms on $M_i$. Therefore, the chain complex
$CF(L_0,L_{01}, L_{12},L_2)$ inherits a filtration given by $\mathcal{A}^{in}$. Recall that the
Floer differential decreases the action functional.

Next, we have a similar filtration on $CF(L_0,L_{02},L_2)$, where we write
$$\mathcal{P}^{out} = \{ (\gamma_{0}, \gamma_{2}) |\ \gamma_i : [0,1] \to M_i,
\gamma_{0}(0) \in L_0, \ (\gamma_{0}(1), \gamma_{2}(0)) \in L_{02}, \gamma_{2}(1) \in
L_2   \} $$ and $\mathcal{A}^{out}$ is defined as before. Note that $\mathcal{P}^{out}$
consists of elements of $\mathcal{P}^{in}$ such that $\gamma_1$ is constant.
Therefore, the action functional $\mathcal{A}^{out} = \mathcal{A}^{in}$ whenever both are defined. 
 
In view of the fact that constant maps are the only zero dimensional solutions which preserve the action (and they are transversely cut out as proved in Proposition \ref{transverse}), to conclude that $\Phi$ is
an isomorphism, it suffices to show that $\Phi$ is a filtered chain map. For this, it
suffices to show that if $\mathcal{M}^0_J(\underbar{x}, \underbar{y})$ is non-empty,
then the following inequality holds : $$\mathcal{A}^{in}(\underbar{x}) \geq
\mathcal{A}^{out}(\underbar{y}) $$
where the equality holds only if $\underbar{x} = \underbar{y}$ and $\mathcal{M}^0_J(\underbar{x}, \underbar{y})$ consists entirely of the trivial solution. 

To see this, let $u$ be a holomorphic curve in
$\mathcal{M}^0_J(\underbar{x},\underbar{y})$. Now, $u$ can be considered as a path
$(\gamma^t_i)_{i=0}^2$ in the path space $\mathcal{P}^{in}$ such that $\gamma^t_1$
shrinks to a constant path as we get to the Y-end and stays constant until the
outgoing end. Now, since $u$ is a holomorphic map, the action strictly decreases unless $u$ is constant.
This gives the desired inequality: $\mathcal{A}^{in}(\underbar{x})
\geq
\mathcal{A}^{in}(\underbar{y}) = \mathcal{A}^{out}(\underbar{y})$ with equality only if $\underbar{x} =\underbar{y}$ and $u$ is constant. So, we have $ \# \mathcal{M}^0_J(\underbar{x},\underbar{x})=1$ with contributions coming only from constant solutions, which are cut transversely. We conclude that $\Phi$ is an isomorphims, as desired.

\QED

\section{Extensions of the main theorem}

In this section, we discuss the proof of Theorem \ref{isomorphism} under positive and
strongly negative monotonicity assumptions. This result in the positively monotone case was first proved  by Wehrheim and Woodward by different techniques.  However, the strongly negative
monotone case is new and important for our application.  

As a first step, we prove a topological lemma which will allow us to establish an a priori energy bound for pseudoholomorphic curves counted in the moduli space
$\mathcal{M}_J (\underbar{x},\underbar{y})$ used for defining the map $\Phi: CF(L_0,
L_{01} , L_{12} , L_2) \to CF(L_0, L_{02}, L_2)$.

\begin{lemma} \label{homotopy} Let $\underbar{x}, \underbar{y} \in CF(L_0, L_{01}, L_{12}, L_2) \simeq
CF(L_0, L_{02}, L_2)$ be two generalized intersection points. Let
$\mathcal{P}(\underbar{x}, \underbar{y})$ be the space of maps $(-\infty,\infty) \to
\mathcal{P}(L_0,L_{01},L_{12},L_2)$ which asymptotically converge to $\underbar{x}$ and $\underbar{y}$.
Similarly, let $\mathcal{B}(\underbar{x},\underbar{y})$ be the space of smooth maps
that is considered for defining the map $\Phi: CF(L_0, L_{01} , L_{12} , L_2) \to
CF(L_0, L_{02}, L_2)$ (see Section 2.2)
Then there is a natural inclusion map $\mathcal{B}(\underbar{x},\underbar{y}) \hookrightarrow
\mathcal{P}(\underbar{x},\underbar{y})$ which induces an isomorphism:
 \[ \pi_0(\mathcal{P}(\underbar{x},\underbar{y}))\cong
\pi_0(\mathcal{B}(\underbar{x}, \underbar{y})) \]

In particular, any homotopy classes of maps used to define $\Phi: CF(L_0, L_{01} , L_{12} ,
L_2) \to CF(L_0, L_{02}, L_2)$ mapping $\underbar{x}$ to $\underbar{y}$ can be
represented as a concatenation of maps $\Phi = u \# c$, where $u \in
\mathcal{P}(\underbar{x},\underbar{y})$ and $c:CF(L_0, L_{01} , L_{12} , L_2)
\to CF(L_0, L_{02}, L_2)$ is the constant map with value $\underbar{y}$.
\end{lemma} \begin{proof} Recall that the space of paths in the absence of
	Hamiltonian perturbations (which we assume, by an a priori arrangement of
	transversality of $L_0, L_{01}, L_{12}, L_2$ as before) is given by
	$\mathcal{P}(L_0, L_{01},L_{12},L_2) = \{ (\gamma_{1}, \gamma_2,
	\gamma_{3}) | \gamma_i : [0,1] \to M_i ,  \gamma_1(0) \in L_0,
	(\gamma_1(1),\gamma_2(0)) \in L_{01}, (\gamma_2(1),\gamma_3(0)) \in
	L_{12}, \gamma_3(1) \in L_2) \}$. We will denote a path
	$\gamma:(-\infty,\infty) \to \mathcal{P}(L_0, L_{01}, L_{12}, L_2)$ in
	this path space by $ \gamma^s=(\gamma^s_1, \gamma^s_2, \gamma^s_3) \in
	\mathcal{P}(\underbar{x},\underbar{y})$, where $s \in(-\infty,
	\infty)$. Now the space $\mathcal{B}(\underbar{x},\underbar{y})$ can be
	identified as a subspace of $\mathcal{P}(\underbar{x},\underbar{y})$
	where $\gamma^s=(\gamma^s_1,\gamma^s_2, \gamma^s_3) \in
	\mathcal{B}(\underbar{x},\underbar{y})$ if and only if $\gamma^s_2$ is
	a constant with respect to $t$ for $s\geq 1$. More precisely, first
	note that any map in $\mathcal{B}(\underbar{x},\underbar{y})$ can
	be homotoped to be constant around a neighborhood of the Y-end (because
	of the exponential convergence at the Y-end). Now, for a moment, let us
	forget about all the decorations and seam conditions on the domain of
	$\Phi$ as given in Figure \ref{quilt}, we then see a rectangle with a
	middle point (Y-end) removed. Let's identify this rectangle with $\f{R} \times [0,1]$ ($\f{R}$ corresponds to the horizontal direction and $[0,1]$ corresponds to the vertical direction in Figure \ref{quilt}). We arrange so that the Y-end point corresponds to
	$(1, 1/2)$. Let's foliate this rectangle by vertical lines $L_s= \{s\}
	\times [0,1]$. Now, if we look at $\Phi(L_s)$ for $s<1$ then we get $3$
	paths $(\gamma^s_1, \gamma^s_2,\gamma^s_3)$ as restrictions of
	$\Phi|_{L_s}$. For $s>1$, we obtain $(\gamma^s_1, \gamma^s_3)$ as
	restrictions of $\Phi|_{L_s}$ such that $(\gamma^s_1 (1) ,
	\gamma^s_3(0)) \in L_{02}$. We can view this alternatively, as a triple
	of paths $(\gamma^s_1, \gamma^s_2, \gamma^s_3)$ where $\gamma^s_2$ is
	the constant path such that $(\gamma^s_1(1), \gamma^s_2(t)) \in L_{01}$
	and $(\gamma^s_2(t), \gamma^s_3 (0)) \in L_{12}$ (such a path is
	uniquely determined by the composability of $L_{01}$ and $L_{12}$).
	Finally, since we arranged that $\Phi$ is constant near the $Y$-end,
	the triple $\gamma^s = (\gamma^s_1, \gamma^s_2, \gamma^s_3)$ of path of paths
	extends continuously over $s=1$ and we obtain $\gamma^s \in \mathcal{P}(\underline{x},\underline{y})$. Note that $\gamma^s_2$ does not vary with $t$ for $s\geq 1$. 

The desired equivalence of path components can be seen by noting that
any path $ \gamma^s=(\gamma^s_1, \gamma^s_2, \gamma^s_3) \in
\mathcal{P}(\underbar{x},\underbar{y})$ is homotopic to a path which is constant for $s > N$ for some
sufficiently large $N$ by the requirement of convergence as $s \to \infty$. One can
then isotope $\gamma^s$ so that it is constant for $s\geq \frac{1}{2}$ in both $s$ and
$t$. Thus, we have an inverse
map $\pi_0(\mathcal{P}(\underbar{x},\underbar{y})) \to \pi_0(\mathcal{B}(\underbar{x},
\underbar{y}))$ to the map induced by the inclusion map. This
gives the desired isomorphism. 

To see the last part of the statement more explicitly, express any the homotopy class $\rho \in \mathcal{P}(\underbar{x},\underbar{y})$  as
$(\gamma^s_1, \gamma^s_2, \gamma^s_3)$ as above with $\gamma^s_i$ constant for $s>N$.
Now, consider the homotopy $\rho^r$ where $r \in [0,1]$  given by
$(\gamma^{s+rN}_1, \gamma^{s+rN}_2, \gamma^{s+rN}_3)$. Then $\rho^{1}$ is a map in
$\mathcal{B}(\underbar{x},\underbar{y})$ which is constant for $s\geq \frac{1}{2}$, hence it is
a concatenation of $u \in \mathcal{P}(\underbar{x},\underbar{y})$ and the constant map with value $\underbar{y}$ as stated. \end{proof}

We are now ready to prove the extension of Theorem \ref{isomorphism} to the monotone
case:

\subsection{Proof of Theorem \ref{positive}}

We briefly recall from \cite{WW} why the Floer cohomology groups in
consideration are well-defined (independent of the choices, invariant under
Hamiltonian deformations, etc.). Given $\underbar{x}, \underbar{y} \in L_{(0)}
\cap L_{(1)}$, the monotonicity assumptions guarantee that the energy of index
$k$ holomorphic strips $u\in \mathcal{M}^k(\underbar{x}, \underbar{y})$ is
constant. Therefore, by Gromov-Floer compactness it suffices to exclude disk
and sphere bubbles. The monotonicity assumptions ensure that any non-trivial
holomorphic disk must have non-zero Maslov index which excludes disk bubbles
in 0-dimensional moduli space. If one assumes that the Lagrangians are
orientable, the Maslov index at a disk bubble has to be at least 2, which excludes disk bubbles
in 0- and 1-dimensional moduli spaces. However, to have a well-defined Floer
cohomology group we also need to avoid disk bubbles in index 2 moduli spaces,
hence we require the minimal Maslov index for disks to be at least 3 (the
sphere bubbles are handled similarly). 

Now, as before we will consider the map \[\Phi : HF(L_0, L_{01}, L_{12}, L_2) \to
HF(L_0, L_{02}, L_2) \]

which is defined by counting solutions in $\mathcal{M}^0_J(\underbar{x},
\underbar{y})$. Let us first study the compactness property of the moduli space
$\mathcal{M}_J(\underbar{x},\underbar{y})$ under our assumptions. We first need to
establish an area-index relation to have an a priori energy bound so that we can
apply Gromov-Floer compactness. This follows easily from Lemma \ref{homotopy}. Namely,
to compute the index of an element $\Phi \in
\mathcal{M}_J(\underbar{x},\underbar{y})$, we can topologically apply a homotopy as in
Lemma \ref{homotopy} so that $\Phi = u \# c$ where $u$ is a map contributing to the
differential of the chain complex  $CF(L_0, L_{01}, L_{12}, L_2)$ and $c \in
\mathcal{M}_J(\underbar{y},\underbar{y})$ is the \emph{constant} map at
$\underbar{y}$. In Section \ref{Mindex}, we computed the index of $c$ to be equal to
zero. (Indeed, this computation is the non-trivial part of the argument that we are
giving here). Therefore, by excision,  $$\I (\Phi) = \I (u) + \I(c)= \I(u)$$
Now, by the area-index relation for the moduli space that $u$ belongs to (this
follows from the monotonicity assumptions, see \cite[Remark 5.2.3]{WW}), the energy of index $k$
holomorphic strips is constant. Since the energy of $\Phi$ is equal to the energy $u$, it
follows that the energy of index $k$ maps in
$\mathcal{M}_J(\underbar{x},\underbar{y})$ is constant. This last statement is the main output of monotonicity assumptions and it is  what we mean by area-index relation for holomorphic curves in $\mathcal{M}_J(\underbar{x},\underbar{y})$.

In view of the area-index relation for maps in $\mathcal{M}_J(\underbar{x},\underbar{y})$, we have energy bounds on all
trajectories of index $0$ and $1$ and thus Gromov-Floer compactification holds.
Therefore, the compactification includes
broken configurations at the ends, possibly also including the Y-end, and disk and sphere
bubbles. However, as before our monotonicity assumptions ensure that disk and sphere
bubbles do not arise in the compactification of index 0 and 1 moduli spaces. Recall
that $\mathcal{M}^0_J$ is used to define the map $\Phi$ and $\mathcal{M}^1_J$ is used
to check that it is a chain map.

We now need to deal with bubbles at the Y-end. Given a sequence of trajectories
$u_i\in \mathcal{M}^1_J (\underbar{x},\underbar{y}) $ breaking along the Y-end, by Gromov-Floer compactness,
we get in the limit a pair $(u_\infty,\delta)$ where $u_\infty \in \mathcal{M}_J (\underbar{x},\underbar{y}) $ (possibly broken at the incoming and outgoing ends) and $\delta$ is a holomorphic strip with Morse-Bott boundary
conditions along $(L_{01}\times L_{12},L_{02}\times \Delta)$.  Note that $\delta$ can also be broken but that does not affect the argument. Now, if $\delta$ is non-constant, it will have non-zero energy, therefore we have $$E(u_\infty) <
E(u_i)$$ 
By an application of Lemma \ref{disk}, from $\delta$ we obtain a map $\tilde{\delta}$ with $E(\delta) = E(\tilde{\delta})$ which represents a class in $\pi_2(\underbar{M}, L_{01} \times L_{12})$ therefore if $\delta$ is non-constant, by energy-index relation $\tilde{\delta}$ has at least index $2$ since $L_{01}
\times L_{12}$ are assumed to be orientable. By the index computation in Lemma $\ref{disk}$ we conclude that $\delta$ has index at least $2$ if it is non-constant. Again by the energy-index relation
proved above for maps in $\mathcal{M}_J(\underbar{x},\underbar{y})$, this implies $$ \I(u_\infty)
\leq \I(u_i) -2 = -1$$ Since the index is additive, there exists at least one
unbroken holomorphic piece in $u_\infty$ with negative index. However, since
these moduli spaces are cut out transversely, this cannot occur. Therefore,
$\mathcal{M}^0_J$ and $\mathcal{M}^1_J$ cannot have any broken configuration
with a bubble at the Y-end.  This concludes the argument that the map $\Phi$
is well-defined.  

Now, to check that $\Phi$ is an isomorphism, we construct an approximate inverse to
$\Phi$.  Let $$\Psi:HF(L_0,L_{02},L_2) \rightarrow HF(L_0,L_{01},L_{12},L_2)$$ be the
map given by counting index 0 holomorphic maps from the quilt that is obtained by reversing the orientation of the quilt that is used to define
$\Phi$.  Arguments identical to those for $\Phi$, show that $\Psi$ is a chain map. We claim that
$$\Psi\circ \Phi=I+K$$ where $I$ is the identity and $K$ is a nilpotent map.  This will
prove that $\Phi$ is an isomorphism.  The diagonal entries of $\Psi\circ \Phi$ are
obtained by counting pairs of broken trajectories $(u_1,u_2)$ with
$u_1$ starting at a critical point $x$ and $u_2$ ending at the same critical
point.  In addition, $u_1$ has the same endpoint as the starting point of
$u_2$.  By the area-index relation, the only such trajectories of index $0$ are
the constants.  More generally, given a sequence of such broken pairs
$(u_1,u_2)$, $(u_3,u_4)\cdots(u_{N-1},u_{N})$ such that
the endpoint of $u_i$ is the starting point of $u_{i+1}$ and the starting
point of $u_1$ is the same as the endpoint of $u_N$ we have that the only index
0 trajectory is the constant one. Indeed, since index and area are additive, any such trajectory with
nonzero area has positive index. \\\\  
Let $N_0$ denote the number of intersection points, i.e. cardinality of $\mathcal{I}( \underbar{L})$.
Consider a broken trajectory, that is a sequence of holomorphic curves that
contribute to $\Psi\circ \Phi$, of index zero with no constant pairs.  In other words, these are the holomorphic curves that contribute to $K= \Psi \circ \Phi - I$.
Any such trajectory with more than $N_0-1$ pairs must have a repeated
intersection point.  Thus such trajectories do not arise, since they would include a segment which has positive index, as we
just explained.  Now, if $K^k(x)$ is nonzero there must be a broken trajectory consisting of $k$ pairs connecting $x$ to some critical point $y$. This trajectory consists of
non-constant pairs.  This is easily seen by induction.  First, $K(x)$ lies in the span
of critical points connected to $x$ by a non-constant pair.  Suppose that $K^i(x)$
lies in the span of critical points $y_i$ connected to $x$ by a non-constant broken
path of pairs of length $i-1$.  Any nonzero matrix element $\langle y, K(y_i)\rangle $
gives rise to a critical point $y$ connected to $y_i$ by a non-constant pair.  It is
thus connected to $x$ by a non-constant path of pairs of length $i$ as desired.
Therefore, we may conclude that $K^{N_0}=0$ since it is contained in the span of
elements coming from broken pairs of length $N_0$.  This completes the proof that $\Phi$ is an isomorphism. \QED

\subsection{Proof of Theorem \ref{negative}} 

In this case, we follow the same steps as in the positively monotone case. The only
difference is the way we handle various exclusions of bubbles. Namely, we exclude
bubbling by first arranging the transversality for the moduli spaces of \emph{simple}
sphere bubbles and $\emph{simple}$ disk bubbles. (Recall that simple means not multiply-covered). The strongly negative monotonicity
assumptions is the assumption that the expected dimension of these moduli spaces is
negative therefore when transversality holds (which can be arranged by choosing
the almost complex structure $J$ in the target generically), we guarantee that these
moduli spaces are empty. A lemma of McDuff (\cite{MS}, Proposition 2.51) and the
decomposition lemma of Kwon-Oh \cite{kwonOh} and Lazzarini (\cite{laz}) allows us to
lift this to non-simple sphere and disk
bubbles.  More specifically, the lemma of McDuff states that any pseudoholomorphic
sphere factors through a simple pseudoholomorphic sphere, so the existence of the
former one implies the existence of the latter. Similarly, Kwon-Oh and Lazzarini's lemma implies
that the existence of any pseudoholomorphic disk ensures the existence of a simple
pseudoholomorphic disk. Now, recall that the expected dimension of unparameterized
moduli space of spheres in $M_i$ in the homology class $[u]$ is $2(\langle [u],
c_1(TM_i) \rangle + m_i -3)$. As part of the hypothesis, we assumed that this number is negative, in fact we
assumed that this number is strictly less than $-2$ to exclude bubbling in
$\mathcal{M}^k_J (\underbar{x}, \underbar{y})$ , for $k=0,1,2$. This is required to
ensure that the Floer cohomology groups that we are considering are well-defined and independent of the
auxiliary choices. Similarly, to avoid disk bubbles, recall that by the real-analyticity of the seams, any disk bubble in a quilted map can
be seen as a disk bubble in $M_i \times M_{i+1}$ with boundary on $L_{i+1}$ for some
$i$. The expected dimension for unparameterized simple disks in the homology class $u$
is given by $\mu_{L_{i+1}}([u]) +(m_i+m_{i+1}) -3$. We assumed that this number is
strictly less than $-2$ to avoid disk bubbles in $\mathcal{M}^k_J (\underbar{x},
\underbar{y})$, for $k=0,1,2$ for the same reason as before. 

Therefore, these considerations imply that the Floer cohomology groups are
well-defined.  Furthermore, the negative monotonicity assumption gives an area-index
relation as before, which guarantees an a priori energy bound on the moduli space
$\mathcal{M}^k_J (\underbar{x}, \underbar{y})$, hence Gromov-Floer compactness
applies.  Since we excluded the possibility of the sphere and disk bubbled
configurations in the compactification of the moduli spaces $\mathcal{M}^0_J$ and
$\mathcal{M}^1_J$, to finish off the only remaining issue is to exclude the bubbling
at the Y-end. We will follow the notation given in the proof of Theorem
\ref{positive}.  We need to exclude non-constant $\delta$ bubbles. Recall that
$$\delta : \f{R} \times [0,1] \to (\underbar{M} ; L_{01} \times L_{12} , L_{02} \times
\Delta ) $$ is a strip with Lagrangian boundary conditions and at the two ends
converges exponentially to points in the Morse-Bott intersection. Note that we can
always ensure the transversality of the moduli space of such $J$-holomorphic curves  by
choosing our $J$ to be $t$-dependent near the Y-end (cf. \cite{FHS}).
\begin{lemma}
$\I(\delta)\leq 0$.
\end{lemma}
\begin{proof}

	Via the construction given in Lemma \ref{disk} we will relate the index of $\delta$ to that of a disk in $\underbar{M}$ with boundary
on $L_{01}\times L_{12}$.  The desired conclusion will then follow from the
monotonicity assumptions of Theorem \ref{negative}.  Recall from Lemma \ref{disk} that we view $\delta$ as a quilt of maps $$\delta_i:\Ar \times [0,1]\rightarrow M_i, i=1,2,3$$
with cyclic Lagrangian boundary conditions $(L_{01},L_{12},L_{02})$.  We introduce $\delta_4:\Ar\times
[0,1]\rightarrow M_1$ to be the map with $\delta_4(s,t)=b(s)$, where $b(s)$ is the unique
point on $M_1$ with $(\delta_2(1,s),b(s),b(s),\delta_0(0,s))\in L_{01}\times
L_{12}$. This defines a quilted map 
$\delta'$ with four patches $\delta_i$ and seams $(L_{01},L_{12},L_{01},L_{12})$. Note
that we have equality of energies $E(\delta)=E(\delta')$ since the image of $\delta_4$ has zero area.  We fold $\delta'$ to obtain a map  $$\delta'':\Ar\times
[0,1] \rightarrow \underbar{M}$$ with boundary on $(L_{01}\times L_{12},L_{01}\times
L_{12})$.  Alternatively, we may view this as a map $$\delta''':D\rightarrow
\underbar{M}$$ where $D$ is the unit disk and $\delta'''$ has Lagrangian boundary
conditions on $L_{01}\times L_{12}$. Now observe that  $\I(\delta''')=\I(\delta'')$ and by the monotonicity assumptions, $$\I(\delta''') =\mu_{L_{01}\times L_{12}}([\delta'''])+m_0+2m_1+m_2\leq 0$$
To conclude, we know by Lemma \ref{disk} that $\I(\delta) + 2\text{dim}(M_1) = \I(\delta''')$. Putting all this together, we get
$$\I(\delta)+2 \text{dim}(M_1) \leq 0$$
 We conclude that $\I(\delta)\leq 0$ as desired. \end{proof} 

 Since $\delta$ is assumed to be non-constant, it cannot have expected dimension zero since translations contribute one dimension to the moduli space. Therefore, $\I(\delta)<0$.  Such $\delta$ cannot occur in view of the transversality of the moduli space.   
Having excluded bubbling at the Y-end, we argue as in the positively monotone case to
conclude that the map $\Phi$ gives the desired isomorphism.  The crucial point is
again to exclude broken non-constant trajectories with the same endpoints.  While in
the positive monotone case these gave rise to moduli spaces of index greater than
zero, under the negative monotone assumptions, the sum of the expected dimension of
such trajectories is negative.  This means that at least one unbroken trajectory in
the sequence has negative index.  This violates the transversality. \QED


\begin{thebibliography}{99999}
\bibitem{AS} {\bf M\, F Atiyah} {\bf I\, M Singer} {\em The index of elliptic operators. I.} Ann. of Math. {\bf 87} (1968) 484--530.

\bibitem{Ben} {\bf B Charbonneau} {\em Analytic aspects of periodic instantons}, PhD
thesis MIT (2004).
\bibitem{Donald} {\bf S Donaldson} {\em Floer homology groups in {Y}ang-{M}ills theory},
Cambridge University Press.
\bibitem{Floer}{\bf A Floer} {\em An unregularized gradient flow for the symplectic
action} Comm. Pure Appl. Math {\bf 41} (1987) 775--813.  
\bibitem{FHS}  {\bf A Floer} {\bf H Hofer} {\bf D Salamon} {\em
Transversality in elliptic Morse theory for the symplectic action} Duke Math. J. {\bf
80} (1995) 251--292.
\bibitem{HLS} {\bf H Hofer} {\bf V Lizan} {\bf J-C. Sikorav} {\em On genericity for holomorphic curves in four-dimensional almost-complex manifolds} J. of Geom. Anal. {\bf 7} (1998) 149--159. 
\bibitem{IvShev}{\bf S Ivashkovich} {\bf V Shevchishin} {\em Complex curves in
almost-complex manifolds and meromorphic hulls} Schriftenreihe des Graduiertenkollegs
Geometrie und mathematische Physik der Universit\"at Bochum {\bf Heft 36} (1999).
\bibitem{kwonOh}{\bf D Kwon} {\bf Y. G. Oh}, {\em Structure of the image of
(pseudo)-holomorphic discs with totally real boundary conditions} Comm. Anal. Geom.
{\bf 8} (2000), 31--82.
\bibitem{laz} {\bf L Lazzarini} {\em Existence of a somewhere injective
pseudo-holomorphic disc} Geom. Funct. Anal. {\bf 10} (2000) 829--862.
\bibitem{lekili} {\bf Y Lekili} {\em Heegaard Floer homology of broken fibrations over
the circle} (2009) Preprint. arXiv: 0903.1773.
\bibitem{lipshitz} {\bf R Lipshitz} {\em A cylindrical reformulation of Heegaard Floer homology} Geometry \& Topology {\bf 10} (2006) 955--1096.
\bibitem{manwood} {\bf C Manolescu}, {\bf C Woodward} {\em Floer homology on the extended moduli space}, Perspectives in analysis, geometry, and analysis. Progr. Math., 96, Birkhauser-Springer, New York 2012.
\bibitem{MS}{\bf D McDuff}, {\bf D Salamon} {\em $J$-holomorphic curves and symplectic
topology} Amer. Mathematical Society, (1994).
\bibitem{PerutzG}{\bf T Perutz} {\em A symplectic Gysin sequence}, (2008) Preprint.  arXiv:0807.1863v1.
\bibitem{Seidel} {\bf P Seidel} {\em Fukaya categories and Picard-Lefschetz theory}
European Mathematical Society. 
\bibitem{WWsurf} {\bf K Wehrheim}, {\bf C Woodward} {\em Pseudoholomorphic quilts}, J. of Symp. Geom. (to appear).
\bibitem{WW} {\bf K Wehrheim}, {\bf C Woodward} {\em Quilted Floer cohomology},
Geometry \& Topology {\bf 14} (2010), 833--902.
\bibitem{WWc} {\bf K Wehrheim}, {\bf C Woodward} {\em Floer cohomology and geometric
composition of Lagrangian Correspondences}, (2009) Preprint. arXiv:0905.1368. 
\bibitem{WWorient} {\bf K Wehrheim}, {\bf C Woodward} {\em Orientations for
pseudoholomorphic quilts} Preprint, available from
www.math.rutgers.edu/{\raise.17ex\hbox{$\scriptstyle\sim$}}ctw/papers.html.
\bibitem{WWL} {\bf K Wehrheim}, {\bf C Woodward} {\em Quilted Floer trajectories with constant components: corrigendum to ``Quilted Floer cohomolgy''} Geom. Topol. {\bf 16} (2012) 127--154.
\bibitem{WoodG} {\bf C Woodward} {\em Gauged Floer theory of toric moment fibers}.  Geom. Funct. Anal. {\bf 21} (2011) 680--749. 



\end{thebibliography}
\end{document}

\maketitle 
\vspace{-0.5in}

As we explain below, an error in Lemma 11 of \cite{LL} affects the main results of that paper
(Theorems 1, 2 and 3 of \cite{LL}).  In this note we show that, under an additional assumption of
monotonicity for cylinders (assumptions (A1), (A2), and (A3) below for Theorems 1, 2 and 3
respectively), the proofs can be fixed easily and the results remain valid. The additional
monotonicity assumption on cylinders holds for the applications of our results that have appeared in
\cite{lekili} and \cite{manwood}, hence those are not affected.

In Lemma 11 of \cite{LL}, we identified the topological type of a bubble resulting from breaking at the Y-end of our
geometric composition map, Y-map, with a disk mapping to $M^{-}_0 \times M_1 \times M_1 \times
M^{-}_2$ with boundary on $L_{01} \times L^{-}_{12}$, where $L_{01} \subset M^{-}_0 \times M_1$ and
$L^{-}_{12}
\subset M_1 \times M^{-}_2$ are Lagrangian boundary conditions. As a consequence of monotonicity
assumptions on disks, it was argued that this bubble cannot exist. This identification of the topological
type of the bubble at the Y-end is wrong. Indeed, the construction given in Lemma 11 of \cite{LL}
shows that the topological type of this bubble at the Y-end
can be identified with an annulus with one boundary component mapping to $L_{01} \times L^{-}_{12}$ and the other mapping to
$\iota(L_{01} \times L^{-}_{12})$ where $\iota : M^{-}_0 \times M_1 \times M_1 \times M^{-}_2 \to
M^{-}_0 \times M_1 \times M_1 \times M^{-}_2$ is the involution given by switching the two factors
of $M_1$. Here, we add the assumption of monotonicity for these annuli that is needed to run the proofs of the main
theorems given in \cite{LL} in full generality.

\begin{figure}[ht]
\centering
\begin{tikzpicture}[scale=0.8]

\draw (0,-2) -- (6,-2) ;
\draw [line width=0.7mm] (5.65,1.4 ) -- (-0.35,1.4);
\draw [line width=0.7mm] (5.65,-1.4) -- (-0.35,-1.4) ;

\draw (6,2 ) -- (0,2);

\draw [dashed, line width=0.7mm] (0.5,0.2) -- (6.5,0.2);

\draw (6,-2) arc (-90:90:0.5cm and 2cm);
\draw (6,2) arc (90:270:0.5cm and 2cm);
\draw [densely dashed] (0,-2) arc (-90:90:0.5cm and 2cm);
\draw (0,2) arc (90:270:0.5cm and 2cm);
    
\node[] at (2.5,-0.6) {\footnotesize $M_0$};
\node[] at (4,1.8) {\footnotesize $M_1$};
\node[] at (4,-1.8) {\footnotesize $M_2$};
\node[] at (2.5,1.2) {\footnotesize $L_{01}$};
\node[] at (3,0) {\footnotesize $L_{12}$};
\node[] at (3,-1.6) {\footnotesize $L_{02}$};

\end{tikzpicture}
    
    \caption{Cylinder bubble at the Y-end.}
    \label{fig1}
\end{figure}
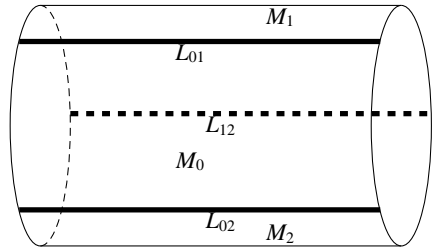

The bubble at the Y-end is a quilted cylinder, depicted in Figure (\ref{fig1}), given by a quilted
map: \[ \delta_i (s,t) : \f{R} \times [0,1] \to M_i , \ \ i=0,1,2 \] with cyclic Lagrangian boundary
conditions $(L_{01}, L_{12}, L^{-}_{02})$. Thus, we have 
\[ (\delta_{0}(s,1), \delta_{1}(s,0)) \in L_{01} \ \ \ , \ \ \ (\delta_{1}(s,1) , \delta_{2}(s,0))
    \in L_{12} \ \ \ , \ \ \ (\delta_{2}(s,1) , \delta_0(s,0)) \in
L^{-}_{02}.  \]

As in \cite{LL}, at the incoming and outgoing ends this converges exponentially to the Morse-Bott
intersection $(L_{01} \times L_{12}) \cap (L_{02} \times \Delta^{-}_{M_1}) \subset M^{-}_0 \times
M_1 \times M^{-}_1 \times M_2$, where $\Delta^{-}_{M_1} \subset M_1 \times M_1^{-}$ is the diagonal
correspondence.  Note that this intersection is diffeomorphic to $L_{02}$ by the assumption that
$L_{02} = L_{01} \circ L_{12}$ is given by an embedded geometric composition. By folding, we can in
turn view this quilted cylinder as a strip: \begin{equation}\label{delta} \delta : \f{R} \times [0,1] \to M^{-}_0 \times M_1
\times M^{-}_1 \times M_2 \end{equation}
with Lagrangian boundary conditions $(L_{01} \times L_{12}, L_{02} \times \Delta_{M^{-}_1})$.

Our main purpose in this note is to rule out bubbles corresponding to $\delta$ maps as above that may appear in
the boundary of the 0- and 1-dimensional components of the moduli space
$\mathcal{M}_{J}(\underline{x}, \underline{y})$ of quilts used in the definition of the Y-map
(see Definition 5 of \cite{LL}) so as to ensure that the Y-map is a chain map. To this end, the
following additional assumptions (A1), (A2), and (A3) must be made in Theorems $1$, $2$ and $3$ of
\cite{LL} respectively:

In the notation of Theorems 1, 2 and 3 of \cite{LL}, let $Z = M^{-}_{r-1} \times M_r
\times M^{-}_r \times M_{r+1}$ with the symplectic form $\omega= (-\omega_{r-1}) \oplus \omega_{r}
\oplus (-\omega_r) \oplus \omega_{r+1} $.

\begin{itemize} 
        \renewcommand{\labelitemi}{(A1)}
    \item (aspherical) For any topological cylinder $\nu : S^1 \times [0,1] \to Z$ with $\nu_{S^1 \times \{0
        \} } \subset L_{r-1} \times L_{r}$ and $\nu_{S^1
        \times\{ 1 \} } \subset (L_{r-1} \circ L_r) \times \Delta_{M_r}$, we have $\int \nu^* \omega =0$.
        \renewcommand{\labelitemi}{(A2)}
     \item ($\tau>0$ positive monotone) For any topological cylinder $\nu : S^1 \times [0,1] \to Z$ with $\nu_{S^1 \times \{0
        \} } \subset L_{r-1} \times L_{r}$ and $\nu_{S^1 \times\{ 1 \} } \subset (L_{r-1} \circ L_r)
        \times \Delta_{M_r}$, we have \[ \int \nu^* \omega = \tau
        I_{Maslov}( \nu^*(L_{r-1}\times L_r), \nu^*((L_{r-1} \circ L_r)\times \Delta_{M_r})) .\]
   \renewcommand{\labelitemi}{(A3)}
   \item ($\tau <0$ strongly negative monotone) For any topological cylinder $\nu : S^1 \times [0,1] \to Z$ with $\nu_{S^1 \times \{0
        \} } \subset L_{r-1} \times L_{r}$ and $\nu_{S^1
        \times\{ 1 \} } \subset (L_{r-1} \circ L_{r}) \times \Delta_{M_r}$ we
        have
        \[ \int \nu^* \omega = \tau
        I_{Maslov}( \nu^*(L_{r-1}\times L_r), \nu^*((L_{r-1} \circ L_r)\times \Delta_{M_r})),\]
        \[\text{and if\ } \int \nu^* \omega > 0, \text{ then }\ \ \ 
    I_{Maslov}( \nu^*(L_{r-1}\times L_r), \nu^*((L_{r-1} \circ L_r) \times \Delta_{M_r})) < -
            \text{dim} (L_{r-1} \circ L_r) + 1. \]
\end{itemize}

In all cases, the relevant cylinders for Lagrangian Floer theory will have the property that
the arc $\nu_{|
\{s\} \times
[0,1]}$ is isotopic relative to its boundary to a point in $(L_{r-1} \times L_r) \cap (L_{r-1} \circ
L_r) \times \Delta_M$ for any, hence all, $s \in S^1$. Henceforth, we will tacitly assume this. 

We continue by letting $M_{r-1} = M_0$, $M_r = M_1$ and $M_{r+1}= M_2$ and $L_{01} = L_{r-1}$,
$L_{12} = L_r$ and $L_{02} = L_{r-1} \circ L_r$.

Note that the limits of $\delta$ at the incoming and outgoing ends has to be in the same path-component
of the intersection $(L_{01} \times L_{12}) \cap (L_{02} \times \Delta^{-}_{M_1})$, since the
possible limits of $\delta$ at the ends are $\{ (m_0,m_1,m_1,m_2) |
(m_0,m_1) \in L_{01}, (m_1,m_2) \in L_{12} \} \cong L_{02}$ and there
    is a seam in $L_{02}$ originally connecting the incoming and outgoing ends
    of the quilted cylinder. Thus, by concatenating an arc that lies in
the Morse-Bott intersection to any strip $\delta$ as in (\ref{delta}), we may assume that $\delta$
converges exponentially to the same point in the Morse-Bott intersection at each of its ends, which in view of removal of singularities (\cite[Lemma 2.5]{wood} or \cite[Sec 2.3]{simcevic}), we may consider as a cylinder $\nu : S^1 \times [0,1]
\to M_0^{-} \times M_1 \times M_1^{-} \times M_2$ with $\nu_{S^1 \times \{0
        \} } \subset L_{01} \times L_{12}$ and $\nu_{S^1
        \times\{ 1 \} } \subset L_{02} \times \Delta_{M^{-}_1}$. 
        The following relations between $\delta$ and $\nu$ are immediate:
\[ \int \delta^* \omega = \int \nu^* \omega \ \ \ , \ \ \ I_{Maslov}(\delta) = I_{Maslov}(\nu) \]
Note that even though the homotopy
        class of the map $\nu$ depends on the choice of the arc in $(L_{01} \times L_{12}) \cap (L_{02}
        \times \Delta^{-}_{M_1})$, the symplectic area and the Maslov index of the resulting
        cylinder is independent of that choice (see \cite{seidel} for the definition of Maslov index in the clean intersection case).

Recall that Fredholm index is given by the sum of Maslov index and the dimension of the Morse-Bott
boundary conditions (cf. Section 2.4 in \cite{LL}). Thus, we have:
\[ index(\delta) = I_{Maslov}(\delta) + \text{dim}(L_{02}) = I_{Maslov}(\nu) + \text{dim}(L_{02}) \] 

With these assumptions in place, the bubbling at the Y-end can be avoided in the situations of
Theorems 1,
2 and 3 of \cite{LL} :

{\bf Correction to proof of Theorem 1:} Assumption (A1) implies that any bubble $\delta$ at the
Y-end must have zero area. In view of holomorphicity of $\delta$, this implies that $\delta$ is
constant.

{\bf Correction to proof of Theorem 2:} Assumption (A2) implies that any bubble $\delta$ at the
Y-end must have positive $I_{Maslov}(\delta)$ in view of holomorphicity of $\delta$, hence
$I_{Maslov}({\delta}) \geq 2$ by orientability of our
Lagrangians. Therefore, the Fredholm index $index(\delta) \geq 2 +
\text{dim}(L_{02})$. 

Now if we have a bubbled configuration $(u, \delta)$ in the index 1 moduli space
$\mathcal{M}^1_{J}(\underline{x},\underline{y})$ of Y-maps, with a bubble at the Y-end, using
the gluing theorem for Fredholm
indices for Morse-Bott ends on the original Y-map one sees that 
\[ 1 = index( (u, \delta) ) = index(u) + index(\delta) - \text{dim}(L_{02}) \]
which implies \[ index(u) \leq -1. \]
Hence, this configuration cannot exist in view of transversality.

{\bf Correction to proof of Theorem 3:} Assumption (A3) implies that any bubble $\delta$ at the
Y-end must have Fredholm index $index(\delta) = I_{Maslov}(\delta) + \text {dim}(L_{02})
<1$. We can always ensure transversality for the moduli space that $\delta$ belongs to by choosing
our $J$ to be $t$-dependent near the Y-end \cite{FHS}. Furthermore, $index(\delta)$
cannot be zero, since translations in the $s$-direction contribute one dimension to the moduli
space. Therefore, we conclude that $index(\delta) < 0$, which violates transversality. 

\begin{remark} In the following two cases of interest, the assumptions (A1), (A2) and (A3)
    introduced above follow, for the relevant types of cylinders, from the original assumptions made
    in Theorems 1, 2 of \cite{LL}, and, for Theorem 3, an index bound for disks in $(M_0^{-} \times
    M_2, L_{02})$. 

\begin{enumerate}
\item Suppose that at least one of the inclusion induced maps $\pi_1(L_{r-1} \times L_r)
    \to \pi_1(Z)$ or $\pi_1((L_{r-1} \circ L_r) \times \Delta_{M_r}) \to \pi_1(Z)$ has torsion image.
    In this case, after possibly taking a cover, an annulus $\nu$ as in (A1) or (A2) can be capped
    off in $Z$ by a disk along one of the Lagrangians to obtain a disk with boundary on the other
    Lagrangian. The assumptions (A1) and (A2) can
    then be deduced from the corresponding assumptions made for disks in Theorems 1 and 2 of \cite{LL}.

The geometric composition isomorphism theorem of \cite{LL} in this case was used in \cite{manwood}.  

\item Suppose that the $\delta_1 : \f{R} \times [0,1] \to M_1$ part of the quilted cylinder is
    homotopic (respecting the boundary conditions) to a strip which is constant along the
    $t$-coordinate. In this case, after homotopy, we can assume that the quilted cylinder
    $(\delta_0, \delta_1, \delta_2)$ is such that $\delta_1$ is independent of $t$. 
    The maps $(\delta_0, \delta_2)$ then give a new quilted cylinder with two parallel seams
        labelled by $L_{02}$. In general, this map can be folded to an annulus in $M^{-}_0 \times
        M_2$ with both boundaries mapping to $L_{02}$. However, since the relevant map has the
        additional property given just below Assumptions (A1-A3), it can also be folded to obtain a disk \[ \tilde{\delta} : D \to M^{-}_0 \times M_2 \]
    with boundary mapping to $L_{02}$ such that 
    \[ Area(\delta) = Area(\tilde{\delta}) \ \ \ , \ \ \ I_{Maslov}(\delta) =
    I_{Maslov}(\tilde{\delta}) \]
    
    The area equality is clear. For the equality of Maslov indices, see the proof of Lemma 2.1.2
    of \cite{WW}. Now, the assumptions (A1), (A2) and (A3) can be deduced from the corresponding
    assumptions for disks in
    $M^{-}_0 \times M_2$ with boundary on $L_{02}$. These follow from assumptions made in Theorem 1
    and 2 in the aspherical and positively monotone cases. In the strongly negative case one has to
    check the strong negativity condition on disks in $M^{-}_0 \times M_2$ with boundary on $L_{02}$
    to deduce (A3). Note that this condition is necessary already in checking that the Floer complex
        $CF(L_0,L_{02},L_2)$ is well-defined. It is a consequence of (A3) in general, and in
        the present case, equivalent to it.
    
    This situation arises, for example, when $L_{12}$ is a fibred coisotropic: it includes into $M_2$ as a coisotropic submanifold and fibres over $M_1$. The geometric composition isomorphism of
    \cite{LL} in this case was used in \cite{lekili} and more recently also in \cite{EL}.  

\end{enumerate} 
\end{remark}
We end with the following remark:
\begin{remark} The quilted cylinder bubbles at the Y-end play a similar role to so-called
    ``figure-eight'' bubbles whose role have been emphasised in \cite{BW} in obstructing the geometric
    composition isomorphism in quilted Floer theory. On the other hand, analytically, the quilted
    cylinder bubbles are better behaved (as they have the standard boundary conditions that approach
    to each other transversely at the ends). In particular, removal of singularities in this case is established in the
    standard way in Floer theory (see for ex. \cite[Lemma 2.5]{wood}). 
\end{remark}

{\bf Acknowledgments:} The error in Lemma 11 of \cite{LL} was pointed out by Katrin Wehrheim. We
    would like to thank Denis Auroux, Jonny Evans and Chris Woodward for helpful comments.

{\it Added in proof:} After this corrigendum was written, a paper by Wehrheim \cite{Wehrheim} appeared which falsely
claims that the application in \cite{lekili} of Theorem 3 is not valid, contrary to Remark 1 above. We do not
endorse this view.